\documentclass{amsart}

\usepackage{amssymb,comment}
\usepackage{mathrsfs}
\usepackage{stmaryrd}

\usepackage{tikz-cd}

\usepackage{mathabx}

\usepackage{lmodern}
\usepackage{csquotes}

\definecolor{darkgreen}{rgb}{0,0.5,0}
\definecolor{darkblue}{rgb}{0,0,0.8}
\definecolor{darkred}{rgb}{0.8,0,0}
\definecolor{lightblue}{rgb}{0,0.6,0.8}

\usepackage[pdfencoding=auto,colorlinks,citecolor=darkgreen,linkcolor=darkblue,urlcolor=darkred]{hyperref}

\usepackage[english]{babel}
\usepackage{url}

\usepackage{multirow}

\usepackage{enumerate}
\usepackage{colonequals}
\usetikzlibrary{matrix, calc, arrows}

\DeclareFontFamily{U}{wncy}{}
\DeclareFontShape{U}{wncy}{m}{n}{<->wncyr10}{}
\DeclareSymbolFont{mcy}{U}{wncy}{m}{n}
\DeclareMathSymbol{\Sha}{\mathord}{mcy}{"58}

\usepackage{microtype}

\theoremstyle{plain}
\newtheorem{theorem}{Theorem}[subsection]
\newtheorem{lemma}[theorem]{Lemma}
\newtheorem{corollary}[theorem]{Corollary}
\newtheorem{proposition}[theorem]{Proposition}
\newtheorem{conjecture}[theorem]{Conjecture}

\newtheorem{example}[theorem]{Example}
\newtheorem{remark}[theorem]{Remark}

\theoremstyle{definition}
\newtheorem{definition}[theorem]{Definition}

\newtheorem{assumption}[theorem]{Assumption}
\newtheorem*{assumption*}{Assumption}
\newtheorem*{claim*}{Claim}

\usepackage{cleveref}
\crefname{theorem}{Theorem}{Theorems}
\crefname{lemma}{Lemma}{Lemmata}
\crefname{corollary}{Corollary}{Corollaries}
\crefname{proposition}{Proposition}{Propositions}
\crefname{definition}{Definition}{Definitions}
\crefname{conjecture}{Conjecture}{Conjectures}
\crefname{question}{Question}{Questions}
\crefname{example}{Example}{Examples}
\crefname{algorithm}{Algorithm}{Algorithms}
\crefname{remark}{Remark}{Remarks}
\crefname{assumption}{Assumption}{Assumptions}

\def\ol#1{\overline{#1}}

\def\Alphabet{A,B,C,D,E,F,G,H,I,J,K,L,M,N,O,P,Q,R,S,T,U,V,W,X,Y,Z}
\def\alphabet{a,b,c,d,e,f,g,h,i,j,k,l,m,n,o,p,q,r,s,t,u,v,w,x,y,z}
\def\endpiece{xxx}
\def\makeAlphabet[#1]{\expandafter\makeA#1,xxx,}
\def\makealphabet[#1]{\expandafter\makea#1,xxx,}
\def\makeA#1,{\def\temp{#1}\ifx\temp\endpiece\else%
	\mkbb{#1}\mkfrak{#1}\mkbf{#1}\mkcal{#1}\mkscr{#1}\mkbs{#1}\expandafter\makeA\fi}%
\def\makea#1,{\def\temp{#1}\ifx\temp\endpiece\else\mkfrak{#1}\mkbf{#1}\mkbs{#1}\expandafter\makea\fi}%
\def\mkbb#1{\expandafter\def\csname bb#1\endcsname{\mathbb{#1}}}
\def\mkfrak#1{\expandafter\def\csname fr#1\endcsname{\mathfrak{#1}}}
\def\mkbf#1{\expandafter\def\csname b#1\endcsname{\mathbf{#1}}}
\def\mkcal#1{\expandafter\def\csname c#1\endcsname{\mathcal{#1}}}
\def\mkscr#1{\expandafter\def\csname s#1\endcsname{\mathscr{#1}}}
\def\mkbs#1{\expandafter\def\csname bs#1\endcsname{{\boldsymbol{#1}}}}
\def\makeop[#1]{\xmakeop#1,xxx,}
\def\mkop#1{\expandafter\def\csname #1\endcsname{{\operatorname{#1}}}} %
\def\xmakeop#1,{\def\temp{#1}\ifx\temp\endpiece\else\mkop{#1}\expandafter\xmakeop\fi}%
\def\makeup[#1]{\xmakeup#1,xxx,}
\def\mkup#1{\expandafter\def\csname #1\endcsname{{\mathrm{#1}\,}}} %
\def\xmakeup#1,{\def\temp{#1}\ifx\temp\endpiece\else\mkup{#1}\expandafter\xmakeup\fi}%
\makeAlphabet[\Alphabet]
\makealphabet[\alphabet]
\makeop[Hom,Tor,Ext,holim,hocolim,dim,Sel,GL,SL,H,ord,Sym,Gal,Ann,ind,Aut,End,cd,Frob,Gr,sp,Tot,sm,an,Pic,NS,Br,tors,Reg,ss,new,cts,ur,alg,pd,disc,cyc,rk,cork,ab,nr,tors,Iw,res,inf,Cl,Quot,corank,ord,length,proj,Hg,mult,add,new,Heeg,str,rel,tr,fs,loc,Quot,Fil,BK,pord,BDP,cond,cores,lcm,Tam,Art,Nm,ab,geo,ari]
\makeup[Spec,Proj,Spwf,Sp,Sh,Spf,Sch,id,dR,rig,cris,im,coker,ac]

\newcommand{\F}{\mathbf{F}}

\newcommand{\Q}{\mathbf{Q}}

\newcommand{\KS}{\mathbf{KS}}

\newcommand{\fm}{\mathfrak{m}}
\newcommand{\fP}{\mathfrak{P}}

\newcommand{\fX}{\mathfrak{X}}
\newcommand{\fN}{\mathfrak{N}}

\newcommand{\eps}{\varepsilon}
\renewcommand{\epsilon}{\varepsilon}
\renewcommand{\theta}{\vartheta}
\renewcommand{\phi}{\varphi}

\newcommand{\mathup}[1]{\text{\textup{#1}}}
\renewcommand{\H} {\ensuremath{\mathup{H}}}

\newcommand{\Char}{\operatorname{char}}

\newcommand{\defeq}{\colonequals}

\newcommand{\isom}{\cong}

\renewcommand{\injlim}{\varinjlim}
\renewcommand{\projlim}{\varprojlim}

\numberwithin{equation}{section}

\begin{document}
	\title[$p$-converse theorems at potentially good ordinary Eisenstein primes]{$p$-converse theorems for elliptic curves\\ of potentially good ordinary reduction\\ at Eisenstein primes}
	
	\author{Timo Keller}
	\address{(T. Keller) Institut für Mathematik, Universität Würzburg, Emil-Fischer-Strasse 30, 97074,
		Würzburg, Germany}
	\address{Rijksuniversiteit Groningen, Bernoulli Institute, Bernoulliborg, Nijenborgh 9, 9747 AG Groningen, The Netherlands}
	
	\email{math@kellertimo.de}
	\urladdr{\url{https://www.timo-keller.de}}
	\thanks{TK was supported by the 2021 MSCA Postdoctoral Fellowship 01064790 -- Ex\-pli\-cit\-Rat\-Points while working on this article.}
	\author{Mulun Yin}
	\address{(M. Yin) University of California Santa Barbara, South Hall, Santa Barbara, CA 93106, USA}
	\email{mulun@ucsb.edu}
	
	\date{\today}
	
	
	\begin{abstract}
		Let $E/\bQ$ be an elliptic curve and $p\geq 3$ be a prime. We prove the $p$-converse theorems for elliptic curves of potentially good ordinary reduction at Eisenstein primes (i.e., such that the residual representation $E[p]$ is reducible) when the $p$-Selmer rank is $0$ or $1$. The key step is to obtain the anticyclotomic Iwasawa Main Conjectures for an auxiliary imaginary quadratic field $K$ where $E$ does not have CM similar to those in~\cite{CGLS} and descent to $\Q$.
		
		As an application we get improved proportions for the number of elliptic curves in quadratic twist families having rank $0$ or $1$.
	\end{abstract}
	
	\maketitle
	
	\tableofcontents
	
	\section*{Introduction}
	
	\renewcommand{\thetheorem}{\Alph{theorem}}
	\subsection{Statement of the main results}
	
	Let $E/\bQ$ be an elliptic curve. The Birch--Swinnerton-Dyer Conjecture predicts that $\ord_{s=1}L(E,s)$, the order of vanishing of the $L$-function of $E$ at $s=1$, should equal $\rk_{\bZ}E(\bQ)$, the rank of its Mordell--Weil group. The celebrated theorem of Gross--Zagier and Kolyvagin states that for $r\in\{0,1\}$,\[\ord_{s=1}L(E,s)=r\Rightarrow \rk_{\bZ}E(\bQ)=r.\]
	In fact, they proved a stronger result. Consider the following conditions:\begin{enumerate}
		\item $\ord_{s=1}L(E,s)=r$;
		\item $\rk_{\bZ}E(\bQ)=r\text{ and }\#\Sha(E/\bQ)<\infty$.
	\end{enumerate}
	What they proved is that (i)$\Rightarrow$(ii) if $r\in\{0,1\}$, where $\Sha(E/\bQ)$ is the Tate--Shafarevich group of $E$, which is conjectured to be always finite.

	Over the past few years, much progress has been made towards a converse to the above theorem. If $p$ is any prime, one can define the $p^\infty$-Selmer groups associated to $E$, denoted by $\Sel_{p^\infty}(E/\bQ)$. Then there is a third condition:\begin{enumerate}
		\item[(iii)] $\corank_{\bZ_p}\Sel_{p^\infty}(E/\bQ)=r$. 
	\end{enumerate}
	(iii) is naturally a consequence of (ii) because the groups fit into an exact sequence\[
	0\to E(\bQ)\otimes_{\bZ}\bQ_p/\bZ_p\to \Sel_{p^\infty}(E/\bQ)\to \Sha(E/\bQ)[p^\infty]\to 0.
	\]
	The implication (iii)$\Rightarrow$(i) is usually called the $p$-converse to Gross--Zagier--Kolyvagin's theorem, which can be formulated in the following way.\begin{conjecture}
		Let $E$ be an elliptic curve defined over $\bQ$ and let $p$ be a prime. Let $r\in\bZ$. Then\[
		\corank_{\bZ_p}\Sel_{p^\infty}(E/\bQ)=r\Rightarrow \ord_{s=1}L(E,s)=r.
		\]
	\end{conjecture}
	
	When $r\in\{0,1\}$, many important cases of the $p$-converse theorems have been essentially obtained, most of which only allow $p$ to be a prime of ordinary reduction or good supersingular reduction for $E$. For example, when $E[p]$ is reducible as a $\Gal(\ol \bQ/\bQ)$-module (the \textit{Eisenstein} case) and does not have the trivial representation as a $\Gal(\ol\bQ_p/\bQ_p)$-subrepresentation, the good ordinary case is treated in~\cite{CGLS}. The restriction on $E[p]$ has later been removed and generalized by the authors in~\cite{KY24} to include the multiplicative case, while $E$ cannot have supersingular reduction in the Eisenstein case.
	
	Our first result is the anticyclotomic Iwasawa Main Conjectures for elliptic curves at potentially good ordinary Eisenstein primes $p\geq 3$. In fact, we formulate and prove things in terms of a bit more general `Heegner pairs' See~\cref{IMCpo}.
	
	Our main results are the following:
	\begin{theorem}
		Let $E$ be an elliptic curve defined over $\bQ$ and let $p>2$ be a prime of potentially good ordinary reduction for $E$. Assume $E[p]$ is reducible. Let $r\in\{0,1\}$. Then\[
		\corank_{\bZ_p}\Sel_{p^\infty}(E/\bQ)=r\Rightarrow \ord_{s=1}L(E,s)=r,
		\]
		and so $\rk_{\bZ}E(\bQ)=r$ and $\#\Sha(E/\bQ)<\infty$.
	\end{theorem}

	\subsection{Method of proof and outline of the paper}\label{method}
	Let $E/\Q$ be an elliptic curve of analytic rank $0$ or $1$ and $K/\Q$ be a \emph{Heegner field} for $E/\Q$, i.e., such that all $p \mid N_E$ split in $K$ and $r_\an(E/K) = 1$.
	A $p$-converse theorem for an elliptic curve $E$ is naturally a consequence of a `Heegner Point type' Main Conjecture of Iwasawa theory. More precisely, let $\Sel_{p^\infty}(E/K)$ and $S_p(E/K)$ be the Selmer groups fitting into the exact sequences
	\[
	0\to E(K)\otimes \Q_p/\bZ_p \to \Sel_{p^\infty}(E/K)\to \Sha_{p^\infty} \to 0,
	\]
	\[
	0\to E(K)\otimes \bZ_p \to S_p(E/K) \to \projlim_n \Sha_{p^n} \to 0.
	\]
	Then it is predicted that $\Sel_{p^\infty}(E/K)\cong \Q_p/\bZ_p\oplus M\oplus M$ for some finite $\bZ_p$-module $M$ such that
	\[
	\length_{\bZ_p}(M)=\length_{\bZ_p}(S_p(E/K)/\bZ_p\cdot \kappa_1)-v_p(\Tam(E/K)),
	\]
	where $\kappa_1\in S_p(E/K)$ is an element of a `Kolyvagin system' coming from Heegner points and $\Tam(E/K)$ is the product of Tamagawa numbers for $E$ over $K$. The Kolyvagin system argument in~\cite{How2004} gives a partial answer to the above conjecture in many cases. Namely, one obtains the desired structure theorem, while being able to show
	\begin{equation}\label{1div}\length_{\bZ_p}(M)\leq\length_{\bZ_p}(S_p(E/K)/\bZ_p\cdot \kappa_1).\end{equation}
	Unfortunately, to obtain a $p$-converse theorem, one would ask for an opposite inequality, which often requires new ideas in different cases.
	
	One could also consider a more general `$\Lambda$-version' (in fact, it is necessary to do so) of the above conjecture for $\Lambda\coloneq \bZ_p\llbracket T\rrbracket$, the anticyclotomic Iwasawa algebra. Namely, if $K_\infty/K$ is the anticyclotomic $\bZ_p$-extension of $K$ with subfields $\{K_n\subset K_\infty\}_{n\in\bN}$ with $[K_n:K]=p^n$, one could consider the `limiting Selmer groups' $\Sel_{p^\infty}(E/K_\infty)\coloneq\injlim\Sel_{p^\infty}(E/K_n)$ and $S_p(E/K_\infty)\coloneq\projlim S_p(E/K_n)$. They will be closely related to certain (dual) Selmer groups denoted by $X\coloneq \Hom(\H^1_{\cF_\Lambda}(K,M_f),\bQ_p/\bZ_p)$ and $\H^1_{\cF_\Lambda}(K,\bT)$ respectively, and one expects that $X$ should be pseudo-isomorphic to $\Lambda\oplus M\oplus M$ where $M$ is a torsion $\Lambda$-module. Moreover, it is conjectured that \[\Char_\Lambda(M)=\Char_\Lambda(\H^1_{\cF_\Lambda}(K,\bT)/\Lambda\cdot\kappa_1^{\Heeg}),\]where $\Char_\Lambda$ denotes the characteristic $\Lambda$-ideal and $\kappa_1^{\Heeg}$ is again an element of a Kolyvagin system coming from Heegner points. Then in the good ordinary case, the Kolyvagin system argument in~\cite{How2004} also proves the structure theorem and shows that $\Char_\Lambda(M)\supset \Char_\Lambda(\H^1_{\cF_\Lambda}(K,\bT)/\Lambda\cdot\kappa_1^{\Heeg})$, a `one-side divisibility' of the Heegner Point Main Conjecture. It was first observed in~\cite{CGLS} that in the Eisenstein case, one could obtain the reversed divisibility by passing to an equivalent form of the Main Conjecture and comparing the so-called `Iwasawa invariants'.
	
	More precisely, in the good ordinary case, there are three steps in establishing the equalities in the Heegner Point Main Conjectures:
	\begin{enumerate}[I]
		\item Proving one divisibility using a Kolyvagin system argument;
		\item Constructing a so-called \textit{Perrin-Riou's regulator map} or \textit{the big logarithm map} that maps a Heegner class to the Bertolini--Darmon--Prasana $p$-adic $L$-function, establishing the equivalence between the Heegner Point Main Conjecture to the \textit{Greenberg's Main Conjecture} involving Greenberg's Selmer groups and $p$-adic $L$-functions. This relation is sometimes referred to as \textit{the explicit reciprocity law};
		\item In the setting of Greenberg's Main Conjecture, comparing the Iwasawa invariants of the algebraic side (for Selmer groups) and the analytic side (for $p$-adic $L$-functions). An equality of the invariants will turn one divisibility coming from the corresponding one in the Heegner Point Main Conjecture into an equality. This in turn implies the equality in Heegner Point Main Conjecture.
	\end{enumerate}
	
	Step III above heavily relies on the Eisenstein assumption. By the reducibility of $E[p]$, there is a short exact sequence \begin{equation}\label{Eis}
	0\to \F_p(\phi)\to E[p]\to \F_p(\psi)\to 0
	\end{equation}
	for some characters $\phi$ and $\psi$. One can relate the Iwasawa invariants of $f$ to those of $\phi$ and $\psi$ on both the algebraic side (in terms of Selmer groups) and the analytic side (in term of $p$-adic $L$-functions). Now the Iwasawa Main Conjectures (\cite{Rubin1991}\&\cite{CW78}. Also~\cite{deShalit}) for the characters bridge the algebraic side and the analytic side. The Heegner Point Main Conjecture, together with Mazur's control theorem, typically yields the $p$-converse theorem.
	
	In the multiplicative reduction case, one proves the Greenberg's Main Conjecture using congruence in a Hida family (\cite{Skinner}). If $f$ is a weight $2$ newform with $p$ dividing its level, there is a Hida family $\{f_k\}$ of higher weight newforms of good reduction `lying above' $f$. By a limiting process, the Greenberg's Main Conjectures for $f_k$'s imply the Greenberg's Main Conjecture for $f$. One then shows the equivalence between the Greenberg's Main Conjecture and the Heegner Point Main Conjecture by studying the exceptional behavior of the big logarithm map for $f$ (\cite{Cas15},~\cite{KY24}). 
	
	In this paper, we extend the Kolyvagin system argument to some new cases in the additive reduction setting, and show that, with carefully defined Selmer groups, one could also formulate and prove Heegner Point Main Conjectures as well as control theorems that would suffice to give a $p$-converse theorem. In doing so, one needs to distinguish between three cases (see~\cref{additive} for a quick review of types of additive reduction):
	\begin{enumerate}[i]
		\item the \textit{potentially good ordinary case};
		\item the \textit{potentially multiplicative case};
		\item the \textit{potentially supersingular case}.
	\end{enumerate}
	The first two cases are both considered as \textit{potentially ordinary} and can be studied together. In fact, if $f$ is a newform of weight $2$ associated to an elliptic curve of potentially ordinary reduction at $p$, then it is a twist of an $p$-ordinary newform with nebentypus, i.e., $f=\tilde{f}\otimes\epsilon$ for some $p$-ordinary $\tilde{f}$ and a finite order character $\epsilon$, and we will work with a \textit{Heegner pair} $(\tilde{f},\chi_\epsilon)$ that is equivalent to $f$.   
	
	In the potentially good ordinary case, we modify the arguments in~\cite{CGLS}, or rather, a generalization in~\cite{KY24}, to obtain an Iwasawa Main Conjecture over an auxiliary imaginary quadratic field $K$ that is not the CM field of $E$ (if $E$ does not have complex multiplication, we could choose any $K$ satisfying certain hypotheses). Roughly speaking, if $\tilde{f}$ has good ordinary reduction, then one simply considers everything twisted by $\epsilon$, and the arguments are not very different from the good ordinary case without twists.
	
	In the potentially multiplicative case, one can hope to work with a twisted Hida family $\{\tilde{f}_k\otimes\epsilon_k\}$ lying above $f$, where $\{\tilde{f}_k\}$ is a Hida family lying above $\tilde{f}$. Then $\tilde{f}_k\otimes\epsilon_k$ are thus automatically twists of good ordinary forms, which makes reasoning analogous to the ordinary case possible. However, due to the lack of tools of studying Perrin-Riou's regulator map, we will not consider the potentially multiplicative case in this work. Nonetheless, many of our results allow $p$ in the level of the ordinary modular forms.
	
	A key step in our arguments is to find a suitable finite extension of $K$ where $E$ obtains good reduction, and show that climbing up the fields only introduce controllable errors.
	
	The paper is organized as follows:
	
	In~\cref{additive}, we review the reduction types of an elliptic curve, with a focus on different cases of additive reduction. We also discuss reduction types of modular forms for our application.
	
	In~\cref{goodIwa}, we introduce Iwasawa theory, which will be the main tool to prove the main results. We also include a structural but comprehensive proof of the $p$-converse theorem in the good ordinary case.
	
	In~\cref{potordIwasa}, we study Iwasawa theory in the potentially good ordinary cases, where we formulate and prove the Heegner Point Main Conjectures, yielding the proofs of the $p$-converse theorems.
	
	\subsection{Relation to previous works}
	On one hand, $p$-converse theorems are consequences of Iwasawa Main Conjectures. When the analytic, hence algebraic rank of $E/\Q$ is $0$ or $1$, the Iwasawa Main Conjectures have been studied by several authors in good ordinary case (\cite{CGLS} for (residually) reducible,~\cite{SU14} for irreducible) and bad multiplicative case (\cite{KY24} for reducible, \cite{Skinner} for irreducible), whether or not a $p$-converse theorem is explicit. 
	
	In this work, we prove the $p$-converse theorems for elliptic curves of potentially good ordinary reduction at additive odd primes in the residually reducible case.
	
	On the other hand, $p$-converse theorems see interesting applications in arithmetic statistics, allowing the distribution of certain \textit{Selmer ranks} to control the ranks of the Mordell--Weil groups. 
	
	\subsection{Eisenstein primes}
	Let $f$ be the weight $2$ newform associated to $E$. In the residually reducible setting, one could take the advantage of the congruence of $f$ to an Eisenstein series $G$ (by the results of~\cite{Kri16}), and obtain a congruence of the $p$-adic $L$-functions of $f$ (constructed in~\cite{BDP13}) to that of $G$. These two $p$-adic $L$-functions $\cL_f$, $\cL_G$ are carefully studied in~\cite{CGLS}. The fact that we could compare $\cL_f$ to $\cL_G$ using congruence without the need to compare interpolation properties is what allows us to finish the argument in the bad reduction case. In the non-Eisenstein case, however, one seems unable to argue without comparing interpolation properties directly, which make the situations much more mysterious.
	
	\subsection{Potentially multiplicative reduction}
	Recall that our proof of the Heegner Point Main Conjectures is divided into three steps: \begin{itemize}
		\item I. Proving one-side divisibility;
		\item II. Establishing the explicit reciprocity law;
		\item III. Comparing Iwasawa invariants
	\end{itemize}
	Steps I and III can be easily extended to the potentially multiplicative case. However, in proving step II we need to use the results from~\cite{JLZ21}, which do no work any more. Indeed, they need to assume $p$ is split in the chosen imaginary quadratic field $K$, while if $f=\tilde{f}\otimes\eps$ is a potentially multiplicative form with $\eps_f=\mathbf{1}=\eps_{\tilde{f}}$, $\epsilon$ must be a quadratic character such that $p\mid\cond(\epsilon)$. Such a character exactly corresponds to imaginary quadratic fields where $p$ does not split.
	
	It is mentioned in~\cite{JLZ21} that a analog of their results which allows non-split $p$ might be achievable by possible extensions of the work in~\cite{AI19}. If so, $p$-converse theorems for potentially multiplicative reduction should also be within reach, which we hope to examine in future work. For future use, some of our results also cover potentially multiplicative case.

	\subsection{Higher weight modular forms}
	The goal of this paper is to prove the $p$-converse theorems for elliptic curves at primes of potentially good ordinary reduction. However, in the potential use of Hida theory to treat the potentially multiplicative case in weight $2$, it is natural to study twists of good ordinary forms of higher weights (corresponding to `Heegner pairs'). It should be mentioned that the `twists' will in general no longer be classical modular forms since these characters will have infinite order. But we can think of these twists as self-dual twist on the level of Galois representations. Some of our arguments still work for general Heegner pairs.
	
	On the other hand, a modular form of weight greater than $2$ of multiplicative reduction is no longer ordinary. Therefore we will not consider those situations.

	\subsection{Applications.} 
	We would get better proportions of elliptic curves in certain families that satisfy the BSD rank conjecture and provide new evidence towards Goldfeld's Conjecture. See~\cref{Goldfeld}.
	
	\subsection{Conventions.}
	For any intermediate field $L$ contained in the algebraic closure $\bar{\Q}$ of $\Q$, we write $G_L$ for its absolute Galois group $\Gal(\bar{\bQ}/L)$. When we say $A$ and $B$ are pseudo-isomorphic, we mean there is a pseudo-isomorphism from $A$ to $B$. Sometimes we need to work with $\bZ_p$-extensions of $K$ and $L$. For $K_\infty$ the anticyclotomic extension of $K$, we set $L_\infty\defeq LK_\infty$. 
	
	\subsection{Acknowledgements} 
	We thank MY's advisor Francesc Castella for his guidance throughout this project. This work is part of MY's forthcoming Ph.D.\ thesis. 
	
	\renewcommand{\thetheorem}{\arabic{section}.\arabic{subsection}.\arabic{theorem}}
	
	\section{Primes of additive reduction}\label{additive}
	In this section we review some fundamental results on the reduction types of elliptic curves and modular forms at a prime $p$. 
	
	\subsection{Reduction types of elliptic curves}
	Let $E$ be an elliptic curve defined over $\bQ$ and $\tilde{E}$ be its reduction modulo $p$. If $\tilde{E}$ is non-singular, we say that $E$ has good reduction at $p$, otherwise $E$ has bad reduction at $p$. Suppose first that $E$ has good reduction at $p$. Let $a_p=p+1-\#\tilde{E}(\F_p)$. We say that $E$ has \textit{(good) ordinary} reduction at $p$ if $p\nmid a_p$, and $E$ is said to have \textit{(good) supersingular} reduction at $p$ otherwise. One knows that if $E$ has supersingular reduction at $p$, then $E[p]$ must be an irreducible $G_\bQ$-module.
	
	Now suppose $E$ has bad reduction. Then $\tilde{E}$ would have a unique singular point $P$. If $P$ is nodal, we say $E$ has \textit{(bad) multiplicative} reduction at $p$. If $p$ is cuspidal, we say $E$ has \textit{(bad) additive} reduction at $p$. In the multiplicative case, we have $a_p=\pm 1$, so multiplicative reduction is also considered ordinary.
	
	The $L$-function of $E$ is formally defined as an Euler product of local factors at every prime:\[
	L(E,s)\coloneq \prod_{p\ \text{good}}(1-a_pp^{-s}+p^{1-2s})^{-1}\cdot \prod_{p\ \text{bad}}(1-a_p p^{-s})^{-1}
	\]
	where $a_p \in\{0,\pm1\}$ depending on the reduction type modulo bad primes $p \mid N$. 
	
	We also defined the conductor $N_{E/\bQ}$ of $E/\bQ$ to be \[N_{E/\bQ}=\prod_p p^{f_p},\]
	where \[f_p=\begin{cases} 
		0 & \text{$E$ has good reduction at $p$} \\
		1 & \text{$E$ has multiplicative reduction at $p$} \\
		2 & \text{$E$ has additive reduction at $p>3$}\\
		2+\delta_p & \text{$E$ has additive reduction at $p=2,3$},
	\end{cases}
	\]where $\delta_p\geq 0$ is a technical constant.
	
	\subsection{Reduction types of modular forms}\label{formtype}
	By modularity theorem, every elliptic curve is associated with a weight $2$ newform $f(z)=\sum_n a_n q^n\in S_2(\Gamma_0(N_f))^{\new}$ where $q=e^{2\pi iz}$ in the sense that their $L$ functions agree:\[
	L(E,s)=L(f,s)\coloneq \sum_n a_n n^{-s}.
	\]
	The conductor $N_E$ of the elliptic curve agrees with the level $N_f$ of the newform. We could then decompose $N_f$ into a product $N_{\mult}N_{\add}$ such that $\ell\mid N_{\mult}$ implies $\ell$ is of multiplicative reduction and $\ell\mid N_{\add}$ implies $\ell$ is of additive reduction. In general, if we have a modular form $f\in S_{k}(\Gamma_0(N))^{\new}$, we say $p$ is a good prime for $f$ if $p\nmid N$, that $p$ is a multiplicative prime for $f$ if $p\| N$ and that $p$ is an additive prime for $f$ if $p^2\mid N$.
	
	Following~\cite{Kato}, we can also talk about potential good reduction for a modular form.
	\begin{definition}[Remark 12.7 in~\cite{Kato}]
		There exists a finite extension $K$ of $\bQ_p$ having the following properties. For any finite place $v$ of $F$ which does not lie over $p$, the representation of $\Gal(\ol K/K)$ on $V_{F_v}(f)$ is unramified. For any finite place $v$ of $F$ which lies over $p$, the representation of $\Gal(\ol K/K)$ on $V_{F_v}(f)$ is crystalline.
	\end{definition}
	
	\subsection{Primes of additive reduction}
	When $E/\bQ$ has additive reduction at a prime $p$ and is viewed as an elliptic curve defined over $\bQ_p$, one could enlarge the field so that $E$ can gain good reduction or multiplicative reduction over some finite extension of $\bQ_p$. For this reason, additive reduction as sometimes referred to as \textit{unstable} reduction. On the other hand, good reduction and multiplicative reduction do not change when extending the ground field, so they are both called \textit{semistable} reduction. 
	
	When $E$ has complex multiplication, it has potentially good reduction everywhere.

	\subsection{Field of semistable reduction}\label{goodfield}
	Most of the time, we do not need to make a specific choice of an extension $L_u/\bQ_p$ where $E$ gains semistable reduction. However, it is convenient to record what can be said about such extensions. We have the following result from~\cite{Conrad}.
	
	\begin{proposition}[Proposition 6.5 in~\cite{Conrad}]
		Pick $N\geq 3$ not divisible by $p$. Then $E$ acquires semistable reduction at all places of the finite Galois splitting field $\bQ(E[N])/\bQ)$ of $A[N]$.
	\end{proposition}
	
	\begin{lemma}
		There is a prime $q \geq 3$ of good reduction for $E$ such that $E$ acquires good reduction over $\Q_p(E[q])/\Q_p$ and such that this field extension has degree coprime to $p$.
	\end{lemma}
	\begin{proof}
		If $p>3$, one can choose $q=3$ and $\Gal(\bQ(E[q])/\bQ)$ is a subgroup of $\GL_2(\F_3)$, which has order $48$. In particular, one can assume there is a finite extension $L/\bQ$ with a place $u\mid p$ so that $E$ gains semistable reduction over $L_u$ with $[L_u:\bQ_p]$ being coprime to $p$.
		
		When $p = 3$, then there is a prime $3 < q \not\equiv 1 \pmod{3}$ of good reduction for $E$ such that $E$ acquires good reduction over $\Q_3(E[q])/\Q_3$. Note that this is a solvable extension with Galois group $G$ isomorphic to a subgroup of $\GL_2(\F_q)$. Hence by~\cite[Theorem~2.14]{KellerStoll2023} and since $\GL_2(\F_q)$ is not solvable for $q \geq 5$, it must be
		\begin{enumerate}[(a)]
			\item contained in a Borel subgroup,
			\item contained in a normalizer of a (split or non-split) Cartan subgroup, or
			\item exceptional, i.e., having projective image in $S_4$.
		\end{enumerate}
		We prove that we can find a prime $q$ as above such that $3 \nmid \#G$.
		\begin{enumerate}[(a)]
			\item In this case $3 \nmid \#G = (q-1)^2\cdot q$ because of our assumption on $q$.
			
			\item The normalizers of a Cartan have order dividing $2(q-1)$ in the split case, which is coprime to $3$ if $q \not\equiv 1 \pmod{3}$, and $2(q+1)$ in the non-split case. But if $q > 2$ is a prime of good ordinary reduction, then according to the proof of~\cite[Lemma~2.32]{KellerStoll2023} the image of Galois $\ol\rho_3(G_{\Q_q})$ contains a non-trivial semisimple split element, so we are not in the normalizer of a non-split Cartan case. But there is such a $q \not\equiv 1 \pmod{3}$: If $E$ is not CM, the density of ordinary primes is $1$, so we are good. If $E$ has CM by an imaginary quadratic field $K$, then we just need $K \ne \Q(\sqrt{-3})$. But if $K = \Q(\sqrt{-3})$, $E$ is isogenous to $y^2 = x^3 - 1$, which has potentially supersingular reduction at $3$.
			
			\item This cannot happen for $q > 5$ a prime of good reduction according to~\cite[Proposition~2.45]{KellerStoll2023}. \qedhere
		\end{enumerate}
	\end{proof}

	\section{Iwasawa Theory}\label{goodIwa}
	In this section, we discuss some important tools from Iwasawa theory. Let $E$ be an elliptic curve over $\Q$ and $f\in S_2(\Gamma_0(N))^\new$ be the weight $2$ newform associated to $E$. Let $p$ be a prime number. Throughout the paper, $K$ denotes an imaginary quadratic field satisfying the following hypotheses:
	\begin{assumption}\label[assumption]{assumption}
		\begin{enumerate}
			\item $p=v\bar{v}$ is split in $K$;
			\item $K$ satisfy the \emph{Heegner hypothesis}, i.e., every prime $\ell \mid N$ is split in $K$;
			\item the discriminant $D_K$ of $K$ is $<-3$ and odd. 
		\end{enumerate}
	\end{assumption}
	
	Let $K_\infty/K$ be the anticyclotomic $\bZ_p$-extension of $K$ and denote by $\Gamma$ its Galois group $\Gamma\coloneq\Gal(K_\infty/K)\cong \bZ_p$. For each $n\in \bN$, let $K_n\subset K_\infty$ be the subfield with $[K_n:K]=p^n$. The Iwasawa algebra $\Lambda\coloneq \bZ_p\llbracket \Gamma\rrbracket$ can be identified with a formal power series ring $\bZ_p\llbracket T\rrbracket$ be sending a topological generator $\gamma\in\Gamma$ to $1+T$.
	
	For a torsion $\Lambda$-module $M$, one knows that there is a \textit{pseudo-isomorphism} (i.e., a $\Lambda$-morphism with finite kernel and cokernel)\[
	M\sim \bigoplus_{i=1}^s\Lambda/(f_i^{k_i}) \oplus \bigoplus_{j=1}^t\Lambda/(p^{l_j})
	\]
	where each $f_i$ is an irreducible \textit{distinguished polynomial} with $f_i\equiv T^{\deg(f_i)}\pmod{p}$. We define the $\lambda,\mu$-invariants of $M$ to be\[
	\lambda(M)\coloneq\sum_{i=1}^s k_i\deg(f_i),
	\]
	\[\mu(M)\coloneq \sum_{j=1}^t l_j.\]
	The \textit{characteristic ideal} $\Char_\Lambda(M)$ of $M$ is the $\Lambda$-ideal generated by the \textit{characteristic polynomial}\[f_\Lambda(M)\coloneq p^{\mu(M)}\prod_{i=1}^s f_i^{k_i}.\]
	
	\begin{example}
		If $M_1,M_2$ are two $\Lambda$-torsion modules with $\Char_\Lambda(M_1)\subset\Char_\Lambda( M_2)$, then $\Char_\Lambda(M_1)=\Char_\Lambda(M_2)$ if and only if $\lambda(M_1)=\lambda(M_2)$ and $\mu(M_1)=\mu(M_2)$.
	\end{example}
	
	\subsection{Selmer structures and Kolyvagin systems}
	To state the Iwasawa Main Conjectures from which our $p$-converse theorem follows, we first need to introduce suitable \textit{Selmer groups}. They generalize the usual Selmer groups $\Sel_{p^\infty}(E/K)$ and $S_p(E/K)$ in the introduction. The following discussions are taken from~\cite[section~3.1]{CGLS}.
	
	\subsubsection{Selmer structures}
	Let $(R,\fm)$ be a complete Noetherian local ring with field of fractions of characteristic $0$ and with finite residue field of characteristic $p$, and let $M$ be a topological $R[G_K]$-module such that the $G_K$-action is unramified outside a finite set of primes. A \textit{Selmer structure} $\cF$ on $M$ is finite set $\Sigma=\Sigma(\cF)$ of places of $K$ containing $\infty$, the primes above $p$ and the places where $M$ is ramified, together with a choice of $R$-submodules (called \textit{local conditions}) $\H^1_\cF(K_w,M)\subset\H^1(K_w,M)$ for every $w\in\Sigma$. The associated \textit{Selmer group} is then defined as\[
	\H^1_\cF(K,M)\coloneq \ker\Bigl\{\H^1(K^\Sigma/K,M)\to \prod_{w\in\Sigma}\frac{\H^1(K_w,M)}{\H^1_\cF(K_w,M)}\Bigr\}.
	\]
	where $K^\Sigma$ is the maximal extension of $K$ unramified outside $\Sigma$.
	
	Some local conditions we will see frequently are the following:
	\begin{itemize}
		\item The \textit{strict} (resp.\ \textit{relaxed}) local condition: $\H^1_{\str}(K_w,M)\coloneq 0$ (resp.\ $\H^1_{\rel}(K_w,M)\coloneq\H^1(K_w,M) $);
		\item The \textit{unramified} local condition: $\H^1_{\nr}(K_w,M) \coloneq \ker(\H^1(K_w,M)\to \H^1(K_w^{\nr},M))$ where $K_w^{\nr}$ is the maximal unramified extension of $K_w$.
	\end{itemize}
	
	When $M$ is unramified at a prime $w \nmid p$, we also called the unramified local condition the \textit{finite} local condition $\H^1_f(K_w,M)\coloneq \H^1_{\nr}(K_w,M)$. The \textit{singular quotient} is defined by the exact sequence\[
	0\to \H^1_f(K_w,M)\to \H^1(K_w,M)\to \H^1_s(K_w,M)\to 0.
	\]
	
	Let $\sL_0\coloneq \sL_0(M)$ be the set of degree $2$ rational primes (that is, inert in $K$) of $\ell\ne p$ at which $M$ is ramified. For $K[\ell]$ the ring class field of conductor $\ell$, we define the \textit{transverse} local condition at $\lambda\mid\ell\in\sL_0$ by\[
	\H^1_{\tr}(K_\lambda,M)\coloneq \ker(\H^1(K[\ell]_\lambda,M)\to \H^1(K_{\lambda'},M))
	\]
	where $K[\ell]_{\lambda'}$ is the completion of $K[\ell]$ at any prime $\lambda$ above $\ell$.
	
	Lastly, given a submodule $N$ (resp.\ quotient) of $M$ and a local condition $\cF$ on $M$, we define the \textit{propagated} local condition on $N$, still denoted by $\cF$, to be the preimage (resp.\ image) of $\H^1_\cF(K_v,M)$ under the natural map\[
	\H^1(K_v,N)\to \H^1(K_v,M)\] (resp.\ $\H^1(K_v,M)\to\H^1(K_v,N)$).
	
	As in~\cite{How2004}, we call a \textit{Selmer triple} $(M,\cF,\sL)$ the data of a Selmer structure $\cF$ on $M$ and a subset $\sL\subset\sL_0$ with $\sL\cap\Sigma(\cF)=\emptyset$. Given a Selmer triple $(M,\cF,\sL)$ and given pairwise coprime integers $a,b,c$ divisible only by primes in $\sL_0$, we define the modified Selmer group $\H^1_{\cF^a_b(c)}(K,M)$ by choosing $\Sigma(\cF^a_b(c))=\Sigma(\cF)\cup\{w\mid abc\}$ and the local conditions\[
	\H^1_{\cF^a_b(c)}(K_\lambda,M)=\begin{cases}
		\H^1(K_\lambda,M) &\text{if }\lambda\mid a\\
		0 & \text{if }\lambda\mid b\\
		\H^1_{\tr}(K_\lambda,M) &\text{if }\lambda\mid c\\
		\H^1_\cF(K_\lambda,M) &\text{if }\lambda\nmid abc\\
	\end{cases}
	\]
	
	\begin{definition}
		Let $\Quot(M)$ denote the \textit{quotient category} of $M$ whose objects are quotients $M/IM$ of $M$ by ideals $I$ of $R$ and whose morphisms from $M/IM$ to $M/I'M$ are the maps induced by scalar multiplication $r\in R$ with $rI\subset I'$.
		
		A local condition functorial over $\Quot(M)$ is called \textit{Cartesian} if for any injective morphism $N\to M$ the local condition $\cF$ on $N$ is the same as the local condition obtained by propagation from $M$ to $N$.
	\end{definition}
	
	\begin{remark}
		By~\cite[Lemma 1.1.9]{MR04}, the unramified local condition is Cartesian.
	\end{remark}
	
	\subsubsection{Kolyvagin systems}\label{Kolysys}
	From now on let $T$ be a compact $R$-module with a continuous linear $G_K$-action that is unramified outside a finite set of primes. For each $\lambda\mid\ell\in\sL_0=\sL_0(T)$, let $I_\ell$ be the smallest ideal containing $\ell+1$ for which the Frobenius element $\Frob_\lambda\in G_{K_\lambda}$ acts trivially on $T/I_\ell T$. By class field theory, $\lambda$ splits completely in the Hilbert class field of $K$, and the $p$-Sylow subgroups of $G_\ell\coloneq\Gal(K[\ell]/K[1])$ and $\bk_\lambda^\times/\bF_\ell^\times$ are identified via the Artin symbol, where $\bk_\lambda$ is the residue field of $\lambda$. Hence by~\cite[Lemma 1.2.1]{MR04}, there is a \textit{finite-singular comparison isomorphism}\[
	\phi_\lambda^\fs\coloneq\H^1_f(K_\lambda,T/I_\ell T)\cong T/I_\ell T\cong \H^1_s(K_\lambda,T/I_\ell T)\otimes G_\ell
	\]
	
	Given a subset $\sL\subset\sL_0$, let $\sN$ denote the set of square-free products of primes $\ell\in\sL$, and for each $n\in\sN$, define\[
	I_n\coloneq\sum_{\ell\mid n}I_\ell\subset R,\ \ \ G_n\coloneq\bigotimes_{\ell\mid n}G_\ell,
	\]
	with the convention that $1\in\sN$, $I_1=0$ and $G_1=\bZ$.
	
	\begin{definition}
		A \textit{Kolyvagin system} for a Selmer triple $(T,\cF,\sL)$ is a collection of classes\[
		\kappa\coloneq\{\kappa_n\in\H^1_{\cF(n)}(K,T/I_n T)\otimes G_n\}_{n\in\sN}
		\]
		such that $(\phi_\lambda^\fs\otimes \mathbf{1})(\loc_\lambda(\kappa_n))=\loc_\lambda(\kappa_{n\ell})$ for all $n\ell\in\sN$.
	\end{definition}
	
	We denote by $\textbf{KS}(T,\cF,\sL)$ the $R$-module of Kolyvagin systems for $(T,\cF,\sL)$.
	
	\subsection{The Iwasawa Main Conjectures}
	In this section we introduce some Iwasawa Main Conjectures that will be needed in the proof of the $p$-converse theorems. We also discuss some strategies in proving known cases and how they could be adapted to new cases.
	
	\subsubsection{The Heegner Point Main Conjecture}
	When $E$ has good ordinary reduction at a prime $p$, several Iwasawa Main Conjectures have been formulated and proved in e.g.~\cite{CGLS}. The one that serves as a key ingredient in the proof of $p$-converse theorems is the following Heegner Point Main Conjecture first formulated by Perrin-Riou~\cite{PR87}.
	
	Fix a modular parametrization\[\pi\colon X_0(N)\to E.\]
	Then the Kummer images of Heegner points on $X_0(N)$ over ring class fields of $K$ of $p$-power conductor give rise to a class $\kappa^\Heeg\in S\coloneq\projlim S_p(E/K_n)$. The group $S$ is naturally a $\Lambda$-module and the class $\kappa^\Heeg$ is known to be non-$\Lambda$-torsion by results of Cornut and Vatsal~\cite{Cor02}, \cite{Vat03}. We put $X\coloneq \Hom_{\bZ_p}(\injlim\Sel_{p^\infty}(E/K_n), \bQ_p/\bZ_p)$.
	
	\begin{conjecture}[The Heegner Point Main Conjecture]\label{HPMC}
		Let $E/\bQ$ be an elliptic curve and $p>2$ be a prime of good ordinary reduction, and let $K$ be an imaginary quadratic field satisfying the Heegner hypothesis. Then both $S$ and $X$ have $\Lambda$-rank one, and\[
		\Char_\Lambda(X_\tors)=\Char_\Lambda(S/\Lambda\cdot\kappa^\Heeg)^2,
		\]
		where $X_\tors$ denote the $\Lambda$-torsion submodule of $X$.
	\end{conjecture}
	
	Under~\cref{assumption}, the above conjecture is now a theorem by combined results of~\cite{CGLS}, \cite{CGS} and~\cite{KY24} in the Eisenstein case. The proof uses a variation of Kolyvagin system arguments systematically studied in~\cite{How2004} (for the `$\supset$'-divisibility) and a comparison of Iwasawa invariants on both sides to turn the divisibility into an equality. To perform the comparison, one needs to go to a different yet equivalent type of Iwasawa Main Conjecture, the \textit{Greenberg's Iwasawa Main Conjecture}.
	
	\subsubsection{The Greenberg's Main Conjecture}\label{GMC}
	Let $\mathfrak{S}_E$ be the modified Selmer group (called the \textit{(Greenberg's) unramified Selmer group}) obtained from $\Sel_{p^\infty}(E/K_\infty)\coloneq\injlim \Sel_{p^\infty}(E/K_n)$ by relaxing (resp. imposing triviality) at places above $v$ (resp. $\ol v$). Let $\fX_E\coloneq\Hom_{\bZ_p}(\mathfrak{S}_E,\bQ_p/\bZ_p)$ be its Pontryagin dual. From the work of Bertolini--Darmon--Prasanna, there is a $p$-adic $L$-function $\cL_E\coloneq \cL_f\in\Lambda^\ur$ interpolating the central values of the $L$-function of $f/K$ twisted by certain characters of $\Gamma$ of infinite order. Here $\Lambda^\nr\coloneq \Lambda\hat{\otimes}_{\bZ_p}\bZ_p^\nr$ where $\bZ_p^\nr$ is the completion of the ring of integers of the maximal unramified extension of $\bQ_p$.
	
	\begin{conjecture}[Greenberg's Iwasawa Main Conjecture]\label{GrMC}
		Let $E/\bQ$ be an elliptic curve and $p>2$ a prime of good ordinary reduction for $E$. Let $K$ be an imaginary quadratic field satisfying the Heegner hypothesis where $p$ splits. Then $\fX_E$ is $\Lambda$-torsion, and\[
		\Char_\Lambda(\fX_E)\Lambda^\nr=(\cL_E)
		\]
		as ideals in $\Lambda^\nr$.
	\end{conjecture}
	As is already observed in~\cite[Appendix A]{Cas17} (see also~\cite[Theorem 5.2]{BCK21}), in the good ordinary setting, the Greenberg's Main Conjecture is equivalent to the Heegner Point Main Conjecture, and is therefore also a theorem now by the results of the aforementioned authors in the Eisenstein case. The proof uses an explicit congruence of $f$ to a certain Eisenstein series $G$ and compares the algebraic side (the Selmer groups for $f$ and those for the two characters appearing in the semisimplification of $\ol\rho_f$) and the analytic side (the interpolation properties of the $p$-adic $L$-function of $G$ (hence that of $f$) to those of the same characters).
	
	\subsection{Prototype: Iwasawa Main Conjectures with good ordinary reduction}\label{Iwasawago}
	In this subsection we introduce the ingredients that go into the proof of~\cref{HPMC} and~\cref{GrMC} in~\cite{CGLS} and subsequent papers. We assume that $E$ is an elliptic curve and $f$ is a weight $2$ newform associated to $E$.
	
	We begin with a proof of one divisibility using Kolyvagin system argument systematically developed in~\cite{How2004} (modified in~\cite{CGLS} in the Eisenstein case). Recall the notations from~\cref{Kolysys}. We start with stating some hypotheses that the Selmer triples $(T,\cF,\sL)$ we will consider shall satisfy.
	\begin{itemize}
		\item[(H.0)] $T$ is a free $R$-module of rank $2$.
		\item[(H.1)] $T/\fm T$ is reducible with $H^0(K,T/\fm T)=0$.
		\item[(H.2)] For every $v\in\Sigma(\cF)$ the local condition $\cF$ at $v$ is Cartesian.
		\item[(H.3)] There is a perfect, symmetric $R$-bilinear pairing \[
		(\cdot ,\cdot )\colon T\times T\to R(1)
		\]which satisfies $(s^\sigma,t^{\tau\sigma\tau^{-1}})=(s,t)^\sigma$ for every $s,t\in T$ and $\sigma\in G_K$. Here $\tau$ is a fixed complex conjugation. 
		We assume that the local condition $\cF$ is its own exact orthogonal complement under the induced local pairing\[
		\langle \cdot,\cdot\rangle_v\colon \H^1(K_v,T)\times\H^1(K_{v^\tau},T)\to R
		\]for every place $v$ of $K$.
	\end{itemize}
	
	The Selmer triple of particular interest is $(T_\alpha,\cF_\ord,\sL_E)$, where
	\begin{itemize}
		\item $\alpha\colon \Gamma\to R^\times$ is a character with values in the ring of integers of a finite extension $\Phi/\bQ_p$ and $R(\alpha)$ is the free $R$-module of rank $1$ on which $G_K$ acts via the projection $G_K\hookrightarrow\Gamma$ composed with $\alpha$;
		\item $T_\alpha\coloneq T_pE\otimes_{\bZ_p}R(\alpha)$ admits a $G_K$-action given by $\rho_\alpha\coloneq\rho_E\otimes \alpha$ where $\rho_E\colon G_\bQ\to \Aut_{\bZ_p}(T_pE)$ gives the $G_\bQ$-action on the Tate module of $E$;
		\item $\cF_\ord$ is the \textit{ordinary} Selmer structure on $V_\alpha\coloneq T_\alpha\otimes\Phi$ defined with $\Sigma(\cF_\ord)=\{w\mid pN\}$ and\[
		\H^1_{\cF_\ord}(K_w,V_\alpha)\coloneq\begin{cases}
			\im(\H^1(K_w,\Fil^+_w(V_\alpha))\to \H^1(K_w,V_\alpha)) &\text{if }w\mid p;\\
			\H^1_\ur(K_w, V_\alpha) &\text{else},
		\end{cases}
		\] 
		where $\Fil^+_w(T_pE)=\ker\{T_pE\to T_p\tilde{E}\}$ is the kernel of reduction at $w$ and $\Fil^+_w(T_\alpha)\coloneq\Fil^+_w(T_pE)\otimes R(\alpha)$, $\Fil^+_w(V_\alpha)\coloneq \Fil^+_w(T_\alpha)\otimes\Phi$.
		
		Let $\cF_\ord$ also denote the Selmer structure on $T_\alpha$ and $A_\alpha\coloneq T_\alpha\otimes \Phi/R\cong V_\alpha/T_\alpha$ by propagating $\H^1_{\cF_\ord}(K_w,V_\alpha)$ under the maps induced by the exact sequence \[
		0\to T_\alpha\to V_\alpha \to A_\alpha\to 0;
		\]
		\item Let $\sL_E\coloneq\{\ell\in\sL_0(T_pE):a_\ell\equiv\ell+1\equiv 0 \text{(mod }p)\}$, where $a_\ell=\ell+1-\vert \tilde{E}(\F_\ell)\vert$, and $\sN=\sN(\sL_E).$
	\end{itemize}
	
	This Selmer triple satisfies the hypotheses $(\H.0)$ to $(\H.4)$ as in~\cite{CGLS}, from which they deduced the following important intermediate result using a Kolyvagin system argument.
	\begin{theorem}[Theorem 3.2.1 in~\cite{CGLS}]\label{middle}
		Assume $E(K)[p]=0$. Suppose $\alpha\ne 1$ and there is a Kolyvagin system $\kappa_\alpha=\{\kappa_{\alpha,n}\}\in \KS(T_\alpha,\cF_\ord,\sL_E)$ with $\kappa_{\alpha,1}\ne 0$. Then $\H^1_{\cF_\ord}(K,T_\alpha)$ has rank one, and there is a finite $R$ module $M_\alpha$ such that\[
		\H^1_{\cF_\ord}(K,A_\alpha)\cong(\Phi/R)\oplus M_\alpha\oplus M_\alpha
		\]with \[\length_R(M_\alpha)\leq\length_R(\H^1_{\cF_\ord}(K,T_\alpha)/R\cdot\kappa_{\alpha,1})+E_\alpha\]for some constant $E_\alpha\in \bZ_{\geq 0}$ depending only on $C_\alpha,T_pE,$ and $\rk_{\bZ_p}(R)$. 
		
		Here $C_\alpha\coloneq\begin{cases}
			v_p(\alpha(\gamma)-\alpha^{-1}(\gamma)) &\alpha\ne\alpha^{-1},\\
			0&\alpha=\alpha^{-1},\\
		\end{cases}$ where $v_p$ is the $p$-adic valuation normalized so that $v_p(p)=1$ and $\gamma\in\Gamma$ is a topological generator.
		
	\end{theorem}
	
	This partial result is an analogue of~\eqref{1div} with an `error term'. As is explained in~\cite{CGLS}, this error term is needed to avoid the use of the classical `big image' assumption (that is, $\rho_E\mid_{G_K}\colon G_K\to \End_{\bZ_p}(T_pE)$ is surjective, which is automatically satisfied in the residually irreducible setting).
	
	To obtain a reversed divisibility using tools from Iwasawa theory, one needs to consider everything `$\Lambda$-adically'. Consider the $\Lambda$-modules \begin{center} $M_E\coloneq (T_pE)\otimes_{\bZ_p}\Lambda^\vee$ and $\bT\coloneq M_E^\vee(1)\cong (T_pE)\otimes_{\bZ_p}\Lambda$\end{center}where the $G_K$-action on $\Lambda^\vee$ is given by thee inverse $\Psi^{-1}$ of the tautological character $\Psi:G_K\twoheadrightarrow\Gamma\hookrightarrow\Lambda^\times$.
	
	For $w$ a prime above $p$, put\begin{center}
		$\Fil^+_w(M_E)\coloneq\Fil^+_w(T_pE)\otimes_{\bZ_p}\Lambda^\vee$ and  $\Fil^+_w(\bT)\coloneq\Fil^+_w(T_pE)\otimes_{\bZ_p}\Lambda$.
	\end{center}
	
	Define the \textit{ordinary} Selmer structure $\cF_\Lambda$ on $M_E$ and $\bT$ by\[
	\H^1_{\cF_\Lambda}(K_w,M_E)\coloneq\begin{cases}
		\im(\H^1(K_w,\Fil^+_w(M_E))\to\H^1(K_w,M_E)) & \text{if }w\mid p,\\
		0&\text{else},
	\end{cases}
	\]
	and\[\H^1_{\cF_\Lambda}(K_w,\bT)\coloneq\begin{cases}
		\im(\H^1(K_w,\Fil^+_w(\bT))\to\H^1(K_w,\bT)) & \text{if }w\mid p,\\
		0&\text{else}.
	\end{cases}
	\]
	Denote by\[
	\cX=\H^1_{\cF_\Lambda}(K,M_E)^\vee=\Hom_{cts}(\H^1_{\cF_\Lambda}(K,M_E),\bQ_p/\bZ_p)
	\]
	the Pontryagin dual of the associated Selmer group $\H^1_{\cF_\Lambda}(K,M_E)$. Also recall the $\sL_E$ defined earlier.
	
	The following theorem, first obtained in~\cite[Theorem 3.4.2]{CGLS} by specializing at all height one primes of $\Lambda$ except at $p$ and $\fP_0\coloneq(\gamma-1)$, is later modified in~\cite[Theorem 6.5.1]{CGS} to hold without the need to invert any prime.
	
	\begin{theorem}[See Theorem 3.0.2 in~\cite{KY24}.]\label{Hgstr}
		Suppose there is a Kolyvagin system $\kappa\in\KS(\bT,\cF_\Lambda,\sL_E)$ with $\kappa_1\ne 0$. Then $\H^1_{\cF_\Lambda}(K,\bT)$ has $\Lambda$-rank one, and there is a finitely generated torsion $\Lambda$-module $M$ such that\begin{enumerate}
			\item[(i)] $\cX\sim\Lambda\oplus M\oplus M,$
			\item[(ii)] $\Char_\Lambda(M)$ divides $\Char_\Lambda(\H^1_{\cF_\Lambda}(K,\bT)/\Lambda\kappa_1).$
		\end{enumerate}
	\end{theorem}
	
	Now using the Heegner point Kolyvagin system constructed in~\cite[Theorem 4.1.1]{CGLS}, one obtains a one-side divisibility of the Heegner Point Main Conjecture under mild hypothesis.\begin{theorem}\label{Hg1div}
		Assume $E(K)[p]=0$. Then $\H^1_{\cF_\Lambda}(K,\bT)$ has $\Lambda$-rank one, and there is a finitely generated torsion $\Lambda$-module $M$ such that\begin{enumerate}
			\item[(i)] $\cX\sim\Lambda\oplus M\oplus M,$
			\item[(ii)] $\Char_\Lambda(M)$ divides $\Char_\Lambda(\H^1_{\cF_\Lambda}(K,\bT)/\Lambda\kappa_1).$
		\end{enumerate}
	\end{theorem}
	
	\begin{remark}
		\begin{itemize}
			\item The assumption $E(K)[p]=0$ (equivalently, the hypothesis $(\H.1)$) is not essential. Since the Iwasawa Main Conjectures are invariant under isogeny, we are content with working with an isogenous curve which satisfies the assumption, thanks to Ribet's Lemma. (See also~\cite[section 0.2, section 1.4]{KY24}.)
			\item If one considers the most general setting as in~\cite{KY24}, the `$\kappa_1$' might be a $p$-power times the Kolyvagin class constructed in~\cite{CGLS}. The following equivalence of the anticyclotomic Main Conjectures will also be slightly different (see~\cite[Theorem 3.0.6, Theorem 3.0.7]{KY24}). For simplicity we do not discuss the most general case and leave the details until Chapter 3. 
		\end{itemize} 
	\end{remark}
	
	To get the reversed divisibility, we now appeal to the Greenberg's Main Conjecture recalled in~\cref{GMC}. As before, let $\cF_\Gr$ denote Greenberg's local conditions, that is, \[\H^1_{\cF_\Gr}(K_w,-)=\begin{cases}
		\H^1(K_w,-)&\text{if }w=v,\\
		0&\text{if }w=\ol v,\\
		\H^1_{\ur}(K_w,-)&\text{else,}
	\end{cases}\] where $-=M_E$ or $\bT$. In general, one should consider the \textit{unramified} Selmer groups introduced in~\cite{KY24}, where the local condition at $\ol v$ is replaced by \[\ker\bigl(\H^1(K_{\ol v},-)\xrightarrow{\res}\H^1(I_{\ol v},-)^{G_{\ol v}/I_{\ol v}}\bigr)\] where $I_{\ol v}$ is the inertia subgroup at $\ol v$. However, it is shown in \textit{loc.\ cit.} that these two Selmer groups generate the same characteristic $\Lambda$-ideals for all relevant modules. Therefore for simplicity we will stick to the Greenberg's Selmer group in this subsection. We first study the comparison between the two types of Iwasawa Main Conjectures.
	
	\begin{proposition}[Proposition 4.2.1 in~\cite{CGLS}]\label[proposition]{equiv}
		Assume that $p=v\ol v$ splits in $K$ and that $E(K)[p]=0$. Then the following statements are equivalent:\begin{itemize}
			\item[(i)] Both $\H^1_{\cF_\Lambda}(K,\bT)$ and $\cX=\H^1_{\cF_\Lambda}(K,M_E)^\vee$ have $\Lambda$-rank one, and the divisibility\[
			\Char_\Lambda(\cX_\tors)\supset\Char_\Lambda(\H^1_{\cF_\Lambda}(K,\bT)/\Lambda\kappa_\infty)^2
			\]
			holds in $\Lambda$.
			\item[(ii)] Both $\H^1_{\cF_{\Gr}}(K,\bT)$ and $\fX_E=\H^1_{\cF_\Gr}(K,M_E)^\vee$ are $\Lambda$-torsion, and the divisibility\[
			\Char_\Lambda(\fX_E)\Lambda^\ur\supset (\cL_E)
			\] holds in $\Lambda^\ur$.
		\end{itemize}
		Moreover, the same result holds for opposite divisibilities.
	\end{proposition}
	
	Here the $\kappa_\infty$ is a class appearing in the equivalence and is closely related to $\kappa_1$ (see~\cite[Remark 4.1.3]{CGLS}). The proof of the equivalence between the two Iwasawa Main Conjectures is carefully explained in~\cite[Appendix A]{Cas17}, where the key is the existence of a Perrin-Riou regulator map (or a `big logarithm map') relating the Heegner point to the BDP $p$-adic $L$-function.\begin{proposition}[Theorem A.1 in~\cite{Cas17}]\label[proposition]{HgtoBDP}
		Under~\cref{assumption}, if $f$ is $p$-ordinary, then there exists an injective $\Lambda^{\ac}$-linear map\[
		\cL_+:\H^1(K_v,\Fil^+\bT)_{\bZ_p^\nr}\to \Lambda^\nr
		\]
		with finite cokernel such that\[
		\cL_+(\res_v(\bz_f))=-\cL_E\cdot \sigma_{-1,v}
		\]where $\sigma_{-1,v}\in\Gamma$ has order two.
	\end{proposition}
	
	Combining~\cref{Hg1div} and~\cref{equiv}, we also get the following one-side divisibility of the Greenberg's Main Conjecture.\begin{corollary}
		Assume that $p=v\ol v$ splits in $K$ and that $E(K)[p]=0$. Then Both $\H^1_{\cF_{\Gr}}(K,\bT)$ and $\fX_E=\H^1_{\cF_\Gr}(K,M_E)^\vee$ are $\Lambda$-torsion, and the divisibility\[
		\Char_\Lambda(\fX_E)\Lambda^\ur\supset (\cL_E)
		\] holds in $\Lambda^\ur$.
	\end{corollary}
	
	The insight in~\cite{CGLS} is that, in the residually reducible case, the residual representation $E[p]$ itself, or rather, the two characters $\phi,\psi$ appearing in the short exact sequence\[
	0\to\F(\phi)\to E[p]\to \F(\psi)\to 0,
	\]should already encode enough arithmetic information of $E$ in terms of Iwasawa invariants. More precisely, they showed that the Iwasawa invariants of $E$ are the sum of those of the two characters $\phi$ and $\psi$, i.e.\begin{itemize}
		\item $\lambda(\fX_E)=\lambda(\fX_\phi)+\lambda(\fX_\psi)+\lambda(\cP_{\phi,\psi,E}),\  \mu(\fX_E)=\mu(\fX_\phi)+\mu(\fX_\psi)=0,$
		\item $\lambda(\cL_E)=\lambda(\cL_\phi)+\lambda(\cL_\psi)+\lambda(\cP_{\phi,\psi,E}),\  \mu(\cL_E)=\mu(\cL_\phi)+\mu(\cL_\psi)=0$.
	\end{itemize}
	Here $\fX_\theta$ are defined similarly as $\fX_f$ by replacing $M_E$ with $M_\theta\coloneq \bZ_p(\theta)\otimes \Lambda^\vee$ where $\theta\colon G_K\to\bZ_p^\times$ is the (Teichmüller lift) of $\phi$ or $\psi$, and $\cL_\theta$ is the associated Katz $p$-adic $L$-function. $\cP_{\phi,\psi,E}$ is a certain factor that cancels on both sides. In~\cite{CGLS} one needs to put some restrictions on the characters, excluding the trivial character in particular. Such comparisons were extended to arbitrary characters in~\cite{KY24}. When one of $\phi$ and $\psi$ is the trivial character (note that one has $\phi\psi=\omega$ the mod $p$ cyclotomic character), the first equation becomes\begin{itemize}
		\item $\lambda(\fX_E)+1=\lambda(\fX_\phi)+\lambda(\fX_\psi)+\lambda(\cP_{\phi,\psi,E})$.
	\end{itemize}
	
	In fact, all the $\mu$-invariants are vanishing in their (and also our) settings by results of~\cite{Hida2010}, which plays a crucial role in their arguments. It is this fact that allows them to interpret the $\lambda$-invariants of $E$ and the characters $\phi$ and $\psi$ as the dimensions of some finite Selmer groups, which makes a direct comparison possible. 
	
	The final ingredient for the proof of the Iwasawa Main Conjectures for $E$ is the following consequence of the Iwasawa Main Conjectures for the characters. It is discussed in detail in~\cite{KY24}.
	\begin{lemma}[\cite{CW78} and~\cite{Rubin1991}. See also~\cite{deShalit}.]
		Let $\theta\colon G_K\to \bZ_p^\times$ be a finite order character. \\If $\theta\ne\mathbf{1}$, then \begin{itemize}
			\item $\mu(\fX_\theta)=\mu(\cL_\theta),$
			\item $\lambda(\fX_\theta)=\lambda(\cL_\theta).$
		\end{itemize}
		If $\theta=\mathbf{1}$, then the equality between $\mu$-invariants still hold, and\begin{itemize}
			\item $\lambda(\fX_\theta)=\lambda(\cL_\theta)+1.$
		\end{itemize}
	\end{lemma}
	
	Thus in any case, one gets the equalities
	\begin{center}
		$\lambda(\fX_E)=\lambda(\cL_E)$ and $\mu(\fX_E)=\mu(\cL_E).$
	\end{center}
	
	We then obtain the Main theorem from~\cite{CGLS}, \cite{CGS} and~\cite{KY24}.
	\begin{theorem}[Iwasawa Main Conjectures for good ordinary reduction. See Theorem 3.0.7, Remark 3.0.8 in~\cite{KY24}]\label[theorem]{IMC}
		Assume that $p=v\ol v$ splits in $K$. Then the following statements hold:\begin{enumerate}
			\item[(HPMC)] Both $\H^1_{\cF_\Lambda}(K,\bT)$ and $\cX=\H^1_{\cF_\Lambda}(K,M_f)^\vee$ have $\Lambda$-rank one, and the equality \begin{equation*}
				\Char_\Lambda(\cX_{\tors})= \Char_\Lambda(\H^1_{\cF_\Lambda}(K,\bT)/\Lambda\kappa_{\infty})^2
			\end{equation*}
			holds in $\Lambda_{ac}$.
			\item[(GrMC)] Both $\H^1_{\cF_{\Gr}}(K,\bT)$ and $\fX_f=\H^1_{\cF_{\Gr}}(K,M_f)^\vee$ are $\Lambda$-torsion, and the equality\begin{equation*}
				\Char_\Lambda(\fX_f)\Lambda^{\nr}=(\cL_f)
			\end{equation*}
			holds in $\Lambda^{\nr}$.
		\end{enumerate}
	\end{theorem}
	
	\subsection{A control theorem}
	Another ingredient we need in proving the $p$-converse theorem is an anticyclotomic control theorem. In the good ordinary case, it is a combination of~\cite[section 3]{JSW2017} (which studies the \textit{anticyclotomic} Selmer group in the anticyclotomic setting) and~\cite[section~3]{Greenberg1999} (which studies the ordinary Selmer group in the cyclotomic setting).
	
	\begin{theorem}\label{controlgo}
		Let $E/\bQ$ be an elliptic curve which has good, ordinary reduction at $p$, then the map\[
		\Sel_{p^\infty}(E/K) \to \Sel_{p^\infty}(E/K_\infty)^\Gamma
		\]
		has finite kernel and cokernel.
	\end{theorem}
	\begin{proof}
		As in~\cite[section 3]{Greenberg1999}, for $M$ an algebraic extension of $K$, let $\cP_E(M)\defeq \prod\frac{\H^1(M_\eta,E[p^\infty])}{\im(\kappa_\eta)}$ denote the codomain of the global-to-local map defining the Selmer group $\Sel_{p^\infty}(E/M)$, where $\eta$ runs through all places of $M$, and let $\cG_E(M)\defeq\im(\H^1(M,E[p^\infty])\to\cP_E(M))$. Then there is a commutative diagram with exact rows\\
		\begin{tikzcd}
			0\ar[r] &\Sel_{p^\infty}(E/K)\ar[r]\ar[d,"s"]& \H^1(K,E[p^\infty]) \ar[r,"\loc_S"]\ar[d,"h"]& \mathcal{G}_E(K)\ar[r]\ar[d,"g"]&0 \\
			0\ar[r]& \Sel_{p^\infty}(E/K_\infty)^\Gamma \ar[r]& \H^1(K_\infty,E[p^\infty])^\Gamma  \ar[r]&\mathcal{G}_E(K_\infty)^\Gamma.&
		\end{tikzcd}\\
		From snake lemma, it suffices to study $\ker(h),\coker(h)$ and $\ker(g)$, the first two of which are easily seen to be finite. To show the finiteness of $\ker(g)$, it suffices to consider the kernel of $r\colon\cP_E(K)\to\cP_E(K_\infty)^\Gamma$, which is a product of the kernels of \[r_w\colon\frac{\H^1(K_w,E[p^\infty])}{\im(\kappa_w)}\to \frac{\H^1(K{_\infty,\eta},E[p^\infty])}{\im(\kappa_\eta)}\] for each place $w$ of $K$, where $\eta$ is any place of $K_\infty$ above $w$.
		
		In the anticyclotomic setting, one cannot use the arguments in~\cite{Greenberg1999} directly for primes not above $p$, since not every prime is finitely decomposed in $K_\infty$. Since the local conditions away from $p$ of the ordinary Selmer groups and those of the anticyclotomic Selmer groups agree (as they are both defined as the unramified local conditions), the arguments in~\cite[Theorem~3.3.7, Cases~1(a)(b), 2(a)(b)]{JSW2017} show that all such $\ker(r_w)$ are finite.
		
		It remains to study $\ker(r_v)$ and $\ker(r_{\ol v})$. The arguments in~\cite[Lemma 3.4]{Greenberg1999} still apply since the assumptions in Theorem 2.4 from \textit{op.\ cit.}, namely that $\Gal(K_{\infty,\eta}/K_w)$ contains an infinite pro-$p$ subgroup and that the inertia subgroup of $\Gal(K_{\infty,\eta}/K_w)$ is of finite index are still satisfied for $w$ equal to both $v$ and $\ol v$ for the anticyclotomic $\bZ_p$-extension $K_\infty$.
	\end{proof}

	\subsection{Proof of the $p$-converse theorems}\label{pCgo}
	We now explain how our $p$-converse theorems would follow from the Heegner Point Main Conjecture, at least in the good ordinary case (see~\cite[Theorem 5.2.1]{CGLS} for a reference). In fact, we only need one divisibility. We first recall a theorem of Kolyvagin, which works for any prime.
	
	\begin{theorem}[Kolyvagin]\label{Kolythm}
		Let $E/\bQ$ be an elliptic curve and $p$ be a prime. Let $r\in\{0,1\}$. Then\[
		\ord_{s=1}L(E/\bQ,s)=r\ \Rightarrow\  \corank_{\bZ_p}\Sel_{p^\infty}(E/\bQ)=r.
		\]
	\end{theorem}
	Here $\Sel_{p^\infty}(E/\bQ)$ is the usual $p^\infty$ Selmer group defined similarly as in~\cref{method} with $K$ replaced by $\bQ$ (that definition makes sense for any number field). 
	
	The proof of the $p$-converse theorem relies on a choice of an auxiliary imaginary quadratic field $K$ over which $E$ does not have CM and the anticyclotomic Iwasawa Main Conjectures holds. 
	\begin{proof}[Proof of the $p$-converse in good ordinary case]
		Again let $r\in\{0,1\}$.
		
		When $\corank_{\bZ_p}\Sel_{p^\infty}(E/\bQ)=r$, we could choose a $K$ which satisfies~\cref{assumption} such that $\ord_{s=1}L(E^K/\bQ,s)=1-r$, where $E^K$ is the twist of $E$ by $K$. We also assume that $E$ does not have CM by $K$. Then by~\cref{Kolythm}, $\corank_{\bZ_p}\Sel_{p^\infty}(E^K/\bQ)=1-r$. It then follows that $\corank_{\bZ_p}\Sel_{p^\infty}(E/K)=1$. \cref{controlgo} then implies that $\Sel_{p^\infty}(E/K_\infty)^\Gamma$ also has $\bZ_p$-corank $1$, or equivalently, $X_\Gamma$ has $\bZ_p$-rank $1$. Indeed, from the discussion after~\cite[Theorem A]{How2004}, there is an pseudo-isomorphism from $\Sel_{p^\infty}(E/K_\infty)^\vee$ to $X$. On the other hand, from~\cref{Hgstr} one has \[\rk_{\bZ_p}X_\Gamma=\rk_{\bZ_p}\Lambda_\Gamma+2\corank_{\bZ_p}M_\Gamma=1+2\corank_{\bZ_p}M_\Gamma,\] so $M_\Gamma$ must be finite.
		
		This is equivalent to $T\nmid f_\Lambda(M)$, so by~\cref{IMC} (in fact, the divisibility `$\subset$' is enough), $T\nmid f_\Lambda(\H^1_{\cF_\Lambda}(K,\bT)/\Lambda\kappa_\infty)$, which in turn means $(\H^1_{\cF_\Lambda}(K,\bT)/\Lambda\kappa_\infty)_\Gamma$ is finite, so $\kappa_\infty$ is non-torsion. 
		
		There is an injection $\H^1_{\cF_\Lambda}(K,\bT)_\Gamma\hookrightarrow S_p(E/K)$ coming from taking first cohomology of the short exact sequence\[
		0\to \bT\xrightarrow{\cdot T}\bT\to T_pE\to 0
		.\]
		It then follows that $\kappa_\infty$ has non-torsion image in $S_p(E/K)$, but by construction the image is a nonzero multiple of $\cores_{K_1/K}(y_1)$ where $y_1$ is the classical Heegner point. Therefore by Gross--Zagier formula $\ord_{s=1}L(E/K,s)=1$. Since $\ord_{s=1}L(E/K,s)=\ord_{s=1}L(E/\bQ,s)+\ord_{s=1}L(E^K/\bQ,s)$, it follows that $\ord_{s=1}L(E/\bQ)=r$.
	\end{proof}
	
	\section{Iwasawa theory of elliptic curves at primes of potentially ordinary reduction}\label{potordIwasa}
	In this section we treat the case where $E/\bQ$ is an elliptic curve and $p$ is a prime of potentially good ordinary reduction for $E$. We let $L$ be a finite extension of $K$ where $E$ gains good reduction, and $K$ is an imaginary quadratic field satisfying~\cref{assumption} as before. Some of our arguments also allow potentially multiplicative reduction.
	 
	To obtain a $p$-converse theorem in the potentially ordinary case, one would naturally hope to argue as in the good ordinary case (see~\cref{pCgo}), namely, proving an Iwasawa Main Conjecture and appealing to a control theorem. However, several definitions no longer make sense in the bad reduction case, including $\Fil^+_w(T_pE)$ which was defined as the kernel of reduction of the Tate module at a place $w\mid p$ in $K$. In particular, one needs to find a reasonable substitute of the ordinary Selmer group.
	
	Note that one could on the other hand always define, regardless of the reduction type, the Bloch--Kato Selmer group $\H^1_\BK(F,V)$ with the following local conditions:\[
	\H^1_\BK(F_w,V)=\begin{cases}
		\H^1_\ur(F_w,V)&v\nmid p\infty,\\
		\ker(\H^1(F_w,V)\to \H^1(F_w,V\otimes_{\bQ_p}\bB_\cris))&w\mid p,\\
		0&w\mid\infty,
	\end{cases}
	\]where $\bB_\cris$ is Fontaine's ring of crystalline periods, $F$ is any number field and $w$ is a prime of $F$, and $V=T_pE\otimes\bQ_p$. One could also defined the Bloch--Kato Selmer groups for $T\defeq T_pE$ and $A\defeq V/T\cong E[p^\infty]$ via propagation in terms of the short exact sequence\[
	0\to T\to V\to A\to 0.
	\]

	The Bloch--Kato Selmer groups $\H^1_\BK(K,A)$ (resp.\ $\H^1_\BK(K,T)$) agree with the classical Selmer groups $\Sel_{p^\infty}(E/K)$ (resp.\ $S_p(E/K)$). When $E$ has ordinary reduction at $p$, they almost agree with the ordinary Selmer groups $\H^1_{\cF_\ord}(K,A)$ (resp.\ $\H^1_{\cF_\ord}(K,T)$) we defined in~\cref{Iwasawago}, which were necessary for our arguments. In fact, when $E$ has ordinary reduction at a place $w$ of $F$ above $p$, the local conditions $\H^1_\BK(F_w,V)$ and $\H^1_\ord(F_w,V)$ agree, where, in the multiplicative case (which is necessarily of weight $2$ by~\cite[Lemma 2.1.2]{Skinner}), the one dimensional subspace $\Fil^+(V)$ is defined so that $V/\Fil^+(V)$ is unramified. It is sometimes inconvenient to work with the Bloch--Kato Selmer group directly. We will therefore define some appropriate Selmer groups that turn out to be the same as the Bloch--Kato ones. Then we formulate and prove the Heegner Point Main Conjecture and a control theorem as in the good ordinary case. The $p$-converse theorem will follow as a consequence. First we need to study Galois representations attached to elliptic curves of potentially ordinary reduction.
	
	\subsection{Potentially ordinary reduction and twists}\label{twists}
	In this subsection, let $f$ be a newform of weight $2$ associated to an elliptic curve $E$. In general, one difficulty in studying Iwasawa theory for $f$ at bad primes lies in the fact that the Euler factor at $p$ is trivial, which makes it hard to obtain local information. For example, in~\cite{CastellaHsieh}, to study generalized Heegner cycles one needs to assume there is a $p$-adic unit root of the Hecke polynomial $T^2-a_p(f)T+p$, which doesn't exist if $a_p(f)=0$. In this subsection we recall the work in~\cite{NekovarSelmerComplexes} which says that if $E$ has potentially ordinary reduction, then $f$ comes from a twist of an ordinary newform $\tilde{f}$. This fact allows us to study the unit $a_p(\tilde{f})$ to get useful information about $f$. We will then use the theory developed in~\cite{JLZ21}, a generalization of~\cite{CastellaHsieh}, to study so-called \textit{Heegner pairs}.
	
	We begin with the following key observation.
	\begin{theorem}[Proposition 12.11.5\,(iv) in~\cite{NekovarSelmerComplexes}]
		Let $E$ be an elliptic curve defined over $\bQ$ and assume $E$ has potentially ordinary reduction at $p$. Let $f$ be the associated weight $2$ newform. Then there is a finite order character $\epsilon$ such that $f\otimes \epsilon^{-1}$ is ordinary at $p$.
	\end{theorem}
	
	Thus there exists a weight $2$ newform $\tilde{f}$ ordinary at $p$ such that $f=\tilde{f}\otimes\epsilon$, in the sense that $a_\ell(f)=a_\ell(\tilde{f})\epsilon(\ell)$ where $a_\ell(f)$ (resp.\ $a_\ell(\tilde{f})$) are the Fourier coefficients of $f$ (resp.\ $\tilde{f}$) for all $(\ell,p\cdot\cond(\epsilon))=1$.
	 
	Now there are $2$ situations:\begin{itemize}
		\item[Case I] $f$ is a potentially good ordinary modular form of weight $2$ for which $\tilde{f}$ has good reduction.
		\item[Case II] $f$ is a potentially multiplicative modular form of weight $2$, in which case $\tilde{f}$ necessarily has multiplicative reduction at $p$.
	\end{itemize}
	
	Since we assume $f$ has trivial nebentypus, $\tilde{f}$ is force to have nebentypus $\epsilon^{-2}$, i.e., $\tilde{f}\in S_2(\Gamma_0(N'),\epsilon^{-2})$, where $N'$ is the level of $\tilde{f}$. Since we assume $f$ has additive reduction at $p$, $\epsilon$ must be ramified at $p$, i.e. $p\mid\cond(\eps)$. In particular, if $\tilde{f}$ also has trivial nebentypus (e.g., when we are in case II), $\eps$ must be a quadratic character ramified at $p$, or equivalently, $\epsilon$ correspond to an imaginary quadratic field $k$ where $p$ does not split.
	
	We see that an elliptic curve $E$ with corresponding weight $2$ form $f$ that has potentially ordinary reduction at $p$ is always associated with a pair $(\tilde{f},\epsilon)$ such that $f=\tilde{f}\otimes\epsilon$ and $\tilde{f}$ is ordinary. We choose $V(f)$ to be the $p$-adic Galois representation attached to a modular form so that when $f$ corresponds to an elliptic curve $E$, $V(f)$ agrees with $V_E\defeq T_pE\otimes\bQ_p$ (i.e., it is dual to Deligne's construction). In particular, the characteristic polynomials of \textit{arithmetic} Frobenius at $\ell\nmid pN$ agree with the Hecke polynomial of $f$. On the other hand, to study a potentially multiplicative form of weight $2$, one would naturally hope to study a twisted Hida family $\{f_k\otimes\epsilon_k\}$ consisting of twists of good ordinary forms. As we will see shortly, the correct thing one should consider are `Heegner pairs', where in some sense corresponds to self-dual twists of the representations attached to modular forms with nebentypus.
	
	Recall that our fixed imaginary quadratic field $K$ satisfies the Heegner hypothesis, i.e., every prime $\ell\mid N$ is split in $K$. Since $N'\mid \lcm(N,\cond(\epsilon)^2)$, we further assume the prime divisors of $\cond(\epsilon)$ are also split in $K$, so that every prime $\ell\mid N'$ is also split in $K$. Hence there exists an ideal $\fN'\subset \cO_K$ such that\[\cO_K/\fN'\cong \bZ/N'\bZ.\]
	
	Choose a Hecke character $\chi$ of infinity type $(0,0)$ whose restriction to $\bA_\bQ^\times$ is $\epsilon^{-2}$ (equivalently, the restriction to $\hat{\cO}_K^\times$ is $\epsilon^2=\eps_{\tilde{f}}^{-1}$. See~\cite[section 4]{BDP13}). Such a Hecke character exists thanks to~\cite[Remark 2.1.2]{JLZ21}. 
	Now the pair $(\tilde{f},\chi)$ is a \textit{Heegner pair} to which one could associate a Heegner point $z_{(\tilde{f},\chi)}$.

	The following results from \textit{op.\ cit.}\ will be adapted. 
	
	\begin{proposition}[Theorem B in~\cite{JLZ21}]\label[proposition]{JLZmain}
		Let $(\tilde{f},\chi)$ be a pair of an $p$-ordinary form $\tilde{f}\in S_2(\Gamma_0(N'),\epsilon^{-2})$ and a Hecke character $\chi$ of finite type $(\fN',\epsilon^{-2})$ and infinity type $(0,0)$. If $p\nmid N'$, then there is a Perrin-Riou regulator map that maps $z_{(\tilde{f},\chi),\infty}$ to the BDP $p$-adic $L$-function $\cL_p(\tilde{f})(\chi)$ up to a unit. 
	\end{proposition}
	
	The above result is a generalization of~\cref{HgtoBDP} to pairs $(\tilde{f},\chi)$, and will be revisited in~\cref{PRreg}. In particular, it covers case (I) above. This will be the key in showing the equivalence of the Heegner point Main Conjecture and the Greenberg's Main Conjecture once suitable definitions and formulations are in place. In fact, the original result from~\cite{JLZ21}
	also covers modular forms of higher weights, where the Hecke characters will also have nonzero infinity type. We will not consider such Heegner pairs in this work.
	
	With the above definition of Heegner pairs, the natural question to ask is how do they relate to our potentially good ordinary form. In general, Hecke characters with fixed central character is not unique (in fact, they can differ by some finite order anticyclotomic characters). The corresponding $G_K$-representations are also different. However, their anticyclotomic theory will not the detect the difference. In what follows, we will specifically construct a Heegner pair whose corresponding $G_K$-representation agrees with the one for our potentially good ordinary form.
	
	\begin{proposition}\label[proposition]{selfdual}
		Let $f=\tilde{f}\otimes\eps$ be a weight $2$ form associated to an elliptic curve of potentially good ordinary reduction at $p$. Then there is a Heegner pair $(\tilde{f},\chi_\eps)$ such that $V(\tilde{f}\otimes\eps)|_{G_K}$ agrees with $V(\tilde{f})|_{G_K}\otimes \chi_{\eps}^{\Gal}$ where $\chi_{\eps}^{\Gal}$ is the Galois character over $G_K$ associated to the Hecke character $\chi_\eps$ over $K$. 
	\end{proposition}
	\begin{proof}
		Consider the following commutative diagram\begin{center}
			\begin{tikzcd}
				\bA_{K,\text{f}}^\times/K^\times\ar[r,"\Art_K"]\ar[d,"\Nm_{K/\Q}"]& \Gal(K^\ab/K)\ar[d,"\theta_{K/\bQ}"]\ar[rd,dashed,"\chi_\eps^\Gal"] &\\
				\bA_{\bQ,\text{f}}^\times/\bQ^\times \ar[r,"\Art_\bQ"]\arrow[bend right=25]{rr}{\eps^{-1}}\arrow[bend left=15]{u}{\res_{K/\bQ}}& \Gal(\bQ^\ab/\bQ)\ar[r,"\eps_0"]&\bC_p^*
		\end{tikzcd}\end{center}
		where $\theta_{K/\bQ}$ is the natural map induced by the inclusions $\bQ\subset K$ and $\bQ^\ab\subset K^\ab$. 
		
		We fix the character $\epsilon^{-1}$ which will determine $\eps_0$ by a choice of the normalization of the Artin maps. We follow the convention in~\cite{JLZ21}, so that uniformizers map to \textit{geometric} Frobenius elements. In other words, for any $\ell\nmid Np$, $\eps_0$ maps $\Frob_\ell^\geo$ to $\eps(\ell)^{-1}$. We claim that $\chi_\eps\defeq \eps^{-1}\circ\Nm_{K/\bQ}$ is the desired Hecke character such that (i) $(\tilde{f},\chi_\eps)$ is a Heegner pair and (ii) $V(f)|_{G_K}=V(\tilde{f}\otimes\eps)|_{G_K}$ agrees with $V(\tilde{f})|_{G_K}\otimes \chi_{\eps}^{\Gal}$.
		
		We first check (i). This follows from the fact that $\chi_\eps|_{\bA_\bQ^\times}=\chi_\eps\circ\res_{K/\bQ}=\eps^{-1}\circ\Nm_{K/\bQ}\circ\res_{K/\bQ}=\eps^{-2}$ since composing norm with restriction introduces a square.
		
		Now we check (ii). First note that $\Gal(K^\ab/K)$ is isomorphic to the profinite completion of $\bA_{K,\text{f}}^\times/K^\times$, so by the universal property of profinite completion, there is a map $\chi_\eps^\Gal$ from $\Gal(K^\ab/K)$ to $\bC_p^*$ such that $\chi_\eps=\chi_\eps^\Gal\circ \Art_K$. This is the $G_K$-character corresponding to $\chi_\eps$ (note that $G_K$-characters factor through $\Gal(K^\ab/K)$). To check the two $G_K$-representations agree, it suffices to check that the traces of the arithmetic Frobenius elements at $\lambda\mid\ell$ in $K$ for a density $1$ set of rational $\ell\nmid pN$ match on both sides. Hence we consider two cases depending on if $\ell$ is either inert or split in $K$.
		
		Case (I): $\ell$ is inert in $K$. In this situation, there is a unique prime $\lambda$ above $\ell$ and the inclusion of $\Frob^\ari_\lambda$ from $G_\lambda$ to $G_\ell$ (the decomposition groups) is $(\Frob^\ari_\ell)^2$. Hence the trace of $\Frob_\lambda^\ari$ under $V(\tilde{f}\otimes\eps)|_{G_K}$ is $(a_\ell\eps(\ell))^2$. On the other hand, recall that $\eps_0$ maps $\Frob_\ell^\geo$ to $\eps(\ell)^{-1}$, so it maps $\Frob_\ell^\ari$ to $\eps(\ell)$ and the trace of $\Frob_\lambda^\ari$ under $\chi_{\eps}^{\Gal}$ will be $\eps(\ell)^2$. Thus the trace of $\Frob_\lambda^\ari$ under $V(\tilde{f})|_{G_K}\otimes \chi_{\eps}^{\Gal}$ is $a_\ell^2\eps(\ell)^2$ which agrees with that for $V(\tilde{f}\otimes\eps)$.
		
		Case (II): $\ell$ splits in $K$. In this situation, there are two places $\lambda_1,\lambda_2$ above $\ell$ and the inclusion of $\Frob_{\lambda_i}^\ari$ from $G_{\lambda_i}$ to $G_\ell$ is $\Frob_\ell^\ari$ for $i=1,2$. As in the previous case, one easily checks that the traces of $\Frob_{\lambda_i}^\ari$ under both $G_K$-representations are $a_\ell\eps(\ell)$.
	\end{proof}
	
	According to~\cite[Remark 2.3.3]{JLZ21}, $V\defeq V(\tilde{f})\otimes \chi_\eps^\Gal$ satisfies $V^\tau\cong V^*(1)$ (note that their choice of $V_p(f)$ is dual to our $V(f)$). By abuse of notation, we will occasionally use $-\otimes\eps$ to mean $-\otimes\chi_\eps^\Gal$ where $-$ can be $V(\tilde{f}),T_{\tilde{f}}$ or $\Fil^{\pm}$ of them. Thus according~\cref{selfdual}, $V(f)=V(\tilde{f}\otimes\eps)=V(\tilde{f})\otimes\eps$ makes sense.
	
	Finally, we note that~\cref{JLZmain} does not cover cases (II). However, the Heegner cycles associated to a Heegner pair is still available, and one could possibly argue as in~\cite[section 5]{KY24} to study (twists of) multiplicative forms using Hida arguments. 
	
	From now on, we assume $p\nmid N'$.
	
	\subsection{Twisted ordinary Selmer group}\label{tword}
	In this subsection, we naturally extend the definitions for a $p$-ordinary form to a Heegner pair using results from the previous subsection. In particular, these Selmer groups also agree with the Bloch--Kato Selmer groups. For our applications, it is sufficient to work with a Heegner pair $(\tilde{f},\chi_\epsilon)$ coming from an elliptic curve of potentially good ordinary reduction as before. However, we choose to work in general contexts where $\tilde{f}$ is any $p$-good ordinary form unless otherwise stated.
	
	Let $(\tilde{f},\chi_\eps)$ be any Heegner pair. From the $p$-ordinarity of $\tilde{f}$, there is a short exact sequence\[
	0\to \Fil^+(V(\tilde{f}))\to V(\tilde{f}) \to \Fil^-(V(\tilde{f})) \to 0
	\]
	such that $\Fil^-(V(\tilde{f}))$ is unramified.
	
	One then obtains the following twisted sequence\[
	0\to \Fil^+(V(\tilde{f}))\otimes\eps \to V(\tilde{f})\otimes\eps \to \Fil^-(V(\tilde{f}))\otimes\eps \to 0,
	\] which will be abbreviated as\[
	0\to \Fil^+(V_\epsilon)\to V_\epsilon\to \Fil^-(V_\epsilon) \to 0
	\] from now. Note that $\Fil^-(V_\epsilon)$ may no longer be unramified at $p$ if $p\mid \cond(\epsilon)$. For example, when $f=\tilde{f}\otimes \epsilon$ is a modular form with CM, we have $V_\epsilon=V(f)$ and one can explicitly choose $\epsilon$ so that $p\mid\cond(\epsilon)$ (see~\cite[section 5]{Müller24}). Even though we generally assume $p\mid \cond(\eps)$ since $(\tilde{f},\chi_\eps)$ should correspond to a modular form with additive reduction at $p$, we do not make it explicit and all the results still hold if $p\nmid\cond(\eps)$.

	Let $\Fil^+(T_\epsilon)\defeq\Fil^+(T_{\tilde{f}})\otimes\epsilon$, where $\Fil^+(T_{\tilde{f}})\defeq T_{\tilde{f}}\cap \Fil^+(V(\tilde{f}))$ and $T_{\tilde{f}}$ is any Galois stable lattice of $V(\tilde{f})$. Since $\tilde{f}$ is ordinary at $p$, when $\tilde{f}$ corresponds to an elliptic curve $E$, $\Fil^+(T_{\tilde{f}})$ is just $\Fil^+_w(T_pE)=\ker(T_pE\to T_p\tilde{E})$ as before. Then we have the natural definitions of the \textit{ordinary Selmer groups twisted by $\epsilon$}, analogous to the ordinary Selmer group.
	
	\begin{definition}
		Let $(\tilde{f},\chi_\epsilon)$ be a Heegner pair and $\epsilon$ be the central character of $\chi$. The \textit{ordinary Selmer group twisted by $\epsilon$ of $V_\epsilon=V(\tilde{f})\otimes \eps$} is given by the local condition\[
		\H^1_{\cF_{\ord,\epsilon}}(K_w,V_\epsilon)\coloneq\begin{cases}
			\im(\H^1(K_w,\Fil^+(V_\epsilon))\to \H^1(K_w,V_\epsilon)) &\text{if }w\mid p;\\
			\H^1_\ur(K_w, V_\epsilon) &\text{else},
		\end{cases}
		\]  
		The \textit{ordinary Selmer group twisted by $\epsilon$} of $T_\epsilon\defeq T_{\tilde{f}}\otimes\epsilon$ and $A_\epsilon\defeq V_\epsilon/T_\epsilon$ are defined via propagation with respect to the short exact sequence\[
		0\to T_\epsilon\to V_\epsilon\to A_\epsilon\to 0
		\]
	\end{definition}
	
	As before, we will actually work with the following Selmer triple $(T_{\alpha,\epsilon},\cF_{ord,\epsilon},\sL_{\tilde{f},\epsilon})$, where\begin{itemize}
		\item $T_{\alpha,\epsilon}=(T_{\tilde{f}}\otimes \epsilon)\otimes R(\alpha)$,
		\item $\cF_{ord,\epsilon}$ is the \textit{ordinary Selmer structure twisted by $\epsilon$} on $V_{\alpha,\epsilon}=T_{\alpha,\epsilon}\otimes\Phi$ defined with $\Sigma(\cF_{ord,\epsilon})=\{w\mid pN\}$ and \[
		\H^1_{\cF_{ord,\epsilon}}(K_w,V_{\alpha,\epsilon})\defeq\begin{cases}
			\im(\H^1(K_w,\Fil^+(V_{\alpha,\epsilon}))\to\H^1(K_w,V_{\alpha,\epsilon}))&\text{if }w\mid p;\\
			\H^1_\ur(K_w,V_{\alpha,\epsilon})&\text{else},
		\end{cases}
		\]
		where $\Fil^+(T_{\alpha,\epsilon})\defeq(\Fil^+(T_{\tilde{f}})\otimes\epsilon)\otimes R(\alpha)$ and $\Fil^+(V_{\alpha,\epsilon})\defeq\Fil^+(T_{\alpha,\epsilon})\otimes\Phi$.\\
		Let $\cF_{ord,\epsilon}$ also denote the Selmer structure on $T_{\alpha,\epsilon}$ and $A_{\alpha,\epsilon}\defeq T_{\alpha,\epsilon}\otimes\Phi/R\isom V_{\alpha,\epsilon}/T_{\alpha,\epsilon}$ by propagation.
		\item $\sL_{\tilde{f},\epsilon}$ is defined in the same way as $\sL_E$ with $T_pE$ replaced by $T_{\tilde{f}}\otimes\epsilon$ and $a_\ell(E)$ replaced by $a_\ell(\tilde{f})\epsilon(\ell))$.
	\end{itemize}
	\begin{remark}
		Recall that if an elliptic curve $E/\bQ$ has potentially ordinary reduction at $p$ with associated weight $2$ newform $f$, then $f=\tilde{f}\otimes\epsilon$ for a Heegner pair $(\tilde{f},\chi_\epsilon)$ and $T_f=T_{\tilde{f}}\otimes\epsilon$. Letting $\Fil^+(T_f)$ denote $\Fil^+(T_\epsilon)$, we formally retrieve the original definition of the Selmer triple for $f$ (or rather, for $E$). One should keep in mind that the filtration is only formally defined and $\Fil^-(V(f))\defeq V(f)/\Fil^+(V(f))$ may not be unramified. It is unramified up to a finite order character, however.
	\end{remark}
	
	In the setting of the previous remark, we now compare the ordinary Selmer groups twisted by $\epsilon$ of $M(f)$ to the Bloch--Kato Selmer groups $\H^1_\BK(K,M(f))$, where $M(f)=V(f),T(f)$ or $A(f)$.
	
	We first consider the Selmer groups for $V(f)=V(\tilde{f})\otimes\eps$. We will show that the local conditions of both Selmer groups differ by finite amount (more precisely, there is a map between them with finite kernel and cokernel). Note that it suffices to consider the places $w=v$ or $\ol v$ above $p$, since away from $p$ the local conditions agree. 
	
	Choose a finite Galois extension $L/K$ (for example, take $L=K(\epsilon)$) such that the localization $L_u$ at a place $u\mid w\mid p$ trivializes $\epsilon$. Consider the following commutative diagram\\
	\begin{tikzcd}
		0\ar[r] &\H^1(K_w,\Fil^+(V(f)_\alpha))\ar[r,"\isom"]\ar[d]& \H^1(L_u,\Fil^+(V(f)_\alpha))^{\Gal(L_u/K_w)} \ar[r]\ar[d]& 0 \\
		0\ar[r]& \H^1(K_w,V(f)_\alpha) \ar[r,"\res_{L_u/K_w}(\isom)"]& \H^1(L_u,V(f)_\alpha)^{\Gal(L_u/K_w)}  \ar[r]&0,
	\end{tikzcd}\\
	where $V(f)_\alpha=V_{\alpha,\epsilon}$ and $\Fil^+(V(f)_\alpha)=\Fil^+(V_{\alpha,\epsilon})$, and the maps are the natural restriction maps. The rows are isomorphisms because the modules in the cohomology groups are vector spaces (an explicit isomorphism can be given by a composition of restrictions and corestrictions). It the follows that \begin{align*}
		&\im(\H^1(K_w,\Fil^+(V(f)_\alpha))\to\H^1(K_w,V(f)_\alpha))\\=&\res_{L_u/K_w}^{-1}(\im(\H^1(L_u,\Fil^+(V(f)_\alpha))\to\H^1(L_u,V(f)_\alpha))^{\Gal(L_u/K_w)})
	\end{align*}
	
	Similarly there is a commutative diagram\\
	\begin{tikzcd}
		&0\ar[d]&0\ar[d]&\\
		&\H^1_\BK(K_w,V(f)_\alpha)\ar[r]\ar[d]&\H^1_\BK(L_u,V(f)_\alpha)^{\Gal(L_u/K_w)}\ar[d] & 
		\\	0\ar[r] &\H^1(K_w,V(f)_\alpha)\ar[r,"\res_{L_u/K_w}(\isom)"]\ar[d]& \H^1(L_u,V(f)_\alpha)^{\Gal(L_u/K_w)} \ar[r]\ar[d]& 0 \\
		0\ar[r]& \H^1(K_w,V(f)_\alpha\otimes \bB_\cris) \ar[r,"\isom"]& \H^1(L_u,V(f)_\alpha\otimes\bB_\cris)^{\Gal(L_u/K_w)}  \ar[r]&0,
	\end{tikzcd}\\
	and it follows that \[
	\H^1_\BK(K_w,V(f)_\alpha)=\res_{L_u/K_w}^{-1}(\H^1_\BK(L_u,V(f)_\alpha)^{\Gal(L_u/K_w)})
	\]
	Since $L_u$ trivializes $\epsilon$, when considered as modules over $L_u$, $M_\epsilon\defeq M_{\tilde{f}}\otimes\epsilon$ is identified $M_{\tilde{f}}$ for $M=V,T$ or $A$, and $V_{\alpha,\epsilon}$ is identified with $V(\tilde{f})_{\alpha}\defeq (T_{\tilde{f}}\otimes R(\alpha))\otimes \Phi$. Furthermore, $\Fil^+(V_{\alpha,\epsilon})$ is identified with $\Fil^+(V(\tilde{f})_{\alpha})\defeq (\Fil^+(T_{\tilde{f}})\otimes R(\alpha))\otimes \Phi$. In particular, \[\im(\H^1(L_u,\Fil^+(V_{\alpha,\epsilon}))\to\H^1(L_u,V_{\alpha,\epsilon}))^{\Gal(L_u/K_w)}\] can be identified with \[\im(\H^1(L_u,\Fil^+(V(\tilde{f})_\alpha))\to \H^1(L_u,V(\tilde{f})_\alpha)^{\Gal(L_u/K_w)}).\]
	By $p$-ordinarity of $\tilde{f}$, the last group is the same as\[
	\H^1_\BK(L_u,V(\tilde{f})_\alpha)^{\Gal(L_u/K_w)},
	\] which is in turn identified with \[
	\H^1_\BK(L_u,V(f)_\alpha)^{\Gal(L_u/K_w)}. 
	\]
	Combining the above identifications, we arrive at the following observation.
	
	\begin{proposition}
		Let $f$ be a weight $2$ potentially $p$-ordinary newform corresponding to a Heegner pair $(\tilde{f},\chi_\epsilon)$ such that $f=\tilde{f}\otimes\epsilon$. The ordinary Selmer group of $f$ of $V(f)$, defined as the ordinary Selmer group twisted by $\epsilon$ for $V(\tilde{f})\otimes\eps$ agrees with the Bloch--Kato Selmer group for $V(f)$.
		
		Consequently, the two types of Selmer groups for $T_{f,\alpha}=T_{\alpha,\epsilon}$ and $A_{f,\alpha}=A_{\alpha,\epsilon}$ also agree, since the local conditions come from propagation. 
	\end{proposition}
	
	\subsection{One divisibility of the Heegner Point Main Conjecture}
	In this subsection, we run a Kolyvagin system argument for the Heegner pair $(\tilde{f},\chi_\epsilon)$ to show one divisibility of the Heegner Point Main Conjecture, which we first formulate. We will keep the conventions from~\cref{goodIwa} and notations from~\cref{tword}.
	
	Let $M_\epsilon=T_\epsilon\otimes_{\bZ_p}\Lambda^\vee$ and $\bT_\epsilon=M_\epsilon^\vee(1)\isom T_\epsilon\otimes_{\bZ_p}\Lambda$ be as before. Put\[\Fil^+(M_\epsilon)\defeq\Fil^+(T_\epsilon)\otimes_{\bZ_p}\Lambda^\vee\text{ and }\Fil^+(\bT_\epsilon)\defeq\Fil^+(T_\epsilon)\otimes_{\bZ_p}\Lambda.\]
	Define the \textit{ordinary Selmer structure $\cF_{\Lambda,\epsilon}$ twisted by $\epsilon$} on $M_\epsilon$ and $\bT_\epsilon$ by\[
	\H^1_{\cF_\epsilon}(K_w,M_\epsilon)\defeq\begin{cases}
		\im(\H^1(K_w,\Fil^+(M_\epsilon))\to\H^1(K_w,M_\epsilon))&\text{if } w\mid p\\
		0&\text{else},
	\end{cases}
	\]
	and\[
	\H^1_{\cF_{\Lambda,\epsilon}}(K_w,\bT_\epsilon)\defeq\begin{cases}
		\im(\H^1(K_w,\Fil^+_u(\bT_\epsilon))\to\H^1(K_w,\bT_\epsilon))&\text{if } w\mid p\\
		0&\text{else},
	\end{cases}
	\]
	
	Denote by \[\cX=\H^1_{\cF_{\Lambda,\epsilon}}(K,M_\epsilon)^\vee=\Hom_\cts(\H^1_{\cF_{\Lambda,\epsilon}}(K,M_\epsilon),\bQ_p/\bZ_p)\]
	the Pontryagin dual of the associated Selmer group $\H^1_{\cF_{\Lambda,\epsilon}}(K,M_\epsilon)$. 
	
	\begin{conjecture}[Heegner Point Main Conjecture]
		$\H^1_{\cF_{\Lambda,\epsilon}}(K,\bT_\epsilon)$ has $\Lambda$-rank one and there is a finitely generated torsion $\Lambda$-module $M$ such that\begin{enumerate}
			\item[(i)] $\cX\sim\Lambda\oplus M\oplus M$,
			\item[(ii)] $\Char_\Lambda(M)=\Char_\Lambda(\H^1_{\cF_{\Lambda,\epsilon}}(K,\bT_\epsilon)/\Lambda\kappa_1)$ where $\kappa_1$ is a class coming from a Heegner point associated to the Heegner pair $(\tilde{f},\chi_\epsilon)$ as in~\cite{JLZ21}.
		\end{enumerate}
	\end{conjecture}
	
	The remaining part of this section is devoted to the proof of the structure theorems and the `$\supset$' divisibility of item (ii) in the above conjecture.
	
	Starting from now, as in~\cite{KY24}, replacing $T_{\tilde{f}}$ by a different Galois stable lattice if necessary, we assume that $\phi|_{G_p}\ne\mathbf{1}$ in the following short exact sequence\[
	0\to \F(\phi) \to \overline{\rho_{\tilde{f}}\otimes\eps}\to \F(\psi) \to 0
	\] where $G_p\subset G_K$ is the decomposition group at $p$. Note that for our self-dual representation $V_\eps$, one knows that $\phi\psi=\omega$, the mod-$p$ cyclotomic character. This assumption is reasonable because one knows the Iwasawa Main Conjectures will be independent from the choice of a lattice, and our choice is possible thanks to Ribet's lemma. A consequence of this assumption is that $\H^0(K,\overline{\rho_{\tilde{f}}\otimes\epsilon})=0$.
	
	\subsubsection{A Kolyvagin system argument}\label{potKoly}
	We first need an analog of~\cref{middle}. We will work with the following Selmer triple $(T_{\alpha,\epsilon},\cF_{\ord,\epsilon},\sL_{\tilde{f},\epsilon})$ introduced in~\cref{tword}.
	
	We first check the hypothesis (H.0)--(H.3) that the above Selmer triple shall satisfy.
	\begin{itemize}
		\item (H.0) is trivially satisfied.
		\item (H.1) is our running hypothesis (to be removed in the end). Specifically, we assume $\H^0(K,\ol{\rho_{\tilde{f}}\otimes \epsilon})=0$.
		\item (H.2) is satisfied because the local conditions on $T_{\alpha,\epsilon}$ come from propagation from $V_{\alpha,\epsilon}$, and hence are Cartesian by~\cite[Lemma 3.7.1]{MR04}.
		\item (H.3) follows from an analog of~\cite[Theorem 2.1.1]{How2004}. One just need to take an extra twist by $\epsilon$ of everything. Also note that our $V_\epsilon$ is self-dual by the comments after~\cref{selfdual}.
	\end{itemize}
	
	It then follows as in the ordinary case that we get the following structure theorem.
	\begin{proposition}[Analog of Lemma 6.1.1 in~\cite{CGS}]\label[proposition]{pointer}
		Assume $\H^0(K,\overline{\rho_{\tilde{f}}\otimes\epsilon})=0$ and suppose there is a Kolyvagin system $\kappa_{\alpha,\epsilon}=\{\kappa_{\alpha,\epsilon,n}\}\in \KS(T_{\alpha,\epsilon},\cF_{\ord,\epsilon},\sL_{\tilde{f},\epsilon})$ with $\kappa_{\alpha,\epsilon,1}\ne 0$. Then $\H^1_{\cF_{\ord,\epsilon}}(K,T_{\alpha,\epsilon})$ has rank one, and there is a finite $R$ module $M_\alpha$ such that\[
		\H^1_{\cF_{\ord,\epsilon}}(K,A_{\alpha,\epsilon})\cong(\Phi/R)\oplus M_\alpha\oplus M_\alpha
		\]with \[\length_R(M_\alpha)\leq\length_R(\H^1_{\cF_{\ord,\epsilon}}(K,T_{\alpha,\epsilon})/R\cdot\kappa_{\alpha,\epsilon,1})+E_\alpha\]for some constant $E_{\alpha,\epsilon}\in \bZ_{\geq 0}$ depending only on $C_\alpha,T_{\alpha,\epsilon},$ and $\rk_{\bZ_p}(R)$. 
	\end{proposition}
	
	\subsubsection{Iwasawa theory}
	Now we argue as in the ordinary case to get a desired divisibility in the Heegner Point Main Conjecture from the above intermediate result~\cref{pointer}. The idea is again to use specialization at height-one primes of $\Lambda$. The proof will be similar to that of~\cite[Theorem 6.5.1]{CGS} based on~\cite{How2004}, and we only explain what is different. 
	
	To make sense of specialization at height-one primes, we first need an analog of~\cite[Lemma 2.2.7]{How2004}. Let $S_\fP$ be the integral closure of $\Lambda/\fP$ and let $\alpha_\fP$ be the character of $G_K$ on $S_\fP$ via $\alpha_\fP:\Gamma\hookrightarrow \Lambda^\times \to S_\fP^\times$. From the construction in section 2.2 of \textit{op.\ cit.}, there are well-defined maps\[
	\bT_\epsilon\to T_{\alpha_\fP,\epsilon}\text{ and }A_{\alpha_\fP,\epsilon}\to M_\epsilon[\fP]
	\] 
	of $G_K$ and $\Lambda$-modules. To slightly ease the notations, we write $\fP$ in place of $\alpha_\fP$ when there is no ambiguity, and we write $\cF_{\fP,\epsilon}$ for the Selmer structure associated to the Selmer triple $(T_{\fP,\epsilon},\cF_{\ord,\epsilon},\sL_{\tilde{f},\epsilon})$ in~\cref{tword}. 
	\begin{lemma}\label[lemma]{Specialize}
		For every height-one prime $\fP\ne p\Lambda$ of $\Lambda$ and every place $w$ of $K$, the induced maps\[
		\H^1_{\cF_{\Lambda,\epsilon}}(K_w,\bT_\epsilon/\fP\bT_\epsilon)\to\H^1_{\cF_{\fP,\epsilon}}(K_w,T_{\fP,\epsilon})
		\]
		\[
		\H^1_{\cF_{\fP,\epsilon}}(K_w,A_{\fP,\epsilon})\to \H^1_{\cF_{\Lambda,\epsilon}}(K_w,M_\epsilon[\fP])
		\]
		have finite kernels and cokernels which are bounded by constants depending only on $[S_\fP:\Lambda/\fP]$.
	\end{lemma}
	\begin{proof}
		We follow the proof in~\cite{How2004}. The cases for places $w$ not above $p$ are treated in~\cite[Lemma 5.3.13]{MR04}, so we assume $w\mid p$. We will show that $\H^0(K_{\infty,w},\Fil^-(A_\epsilon))$ is finite so the cokernel of the first map in~\cite[eq.~(15)]{How2004} is finite.
		
		To bound the cokernel of the third cokernel of that map, we need the same finiteness result. When $\tilde{f}$ has weight greater than $2$, the Galois action on the quotient $\Fil^-(V_\eps)$ of our self-dual representation $V_\eps$ is an unramified character times a non-trivial power of the cyclotomic character which stays non-trivial over $G_{K_{\infty,w}}$ by~\cite[Lemma 1.0.2(ii)]{KY24}, so $\H^0(K_{\infty,w},\Fil^-(V_\eps))=0$. Hence $\H^0(K_{\infty,w},\Fil^-(A_\eps))=0$ is finite.
		
		When $\tilde{f}$ has weight $2$, let $L/K$ be a finite extension that trivializes $\epsilon$ and let $u$ be a place of $L$ above $w$. By $L_\infty$ we mean the composite field $LK_\infty$. Then $L_{u,\infty}$ is a finite extension of $K_{w,\infty}$.
		
		Then $\H^0(L_{u,\infty},\Fil^-(A_\epsilon))$ is naturally identified with $\H^0(L_{u,\infty},\Fil^-(A_{\tilde{f}}))$.To show the latter is finite, we refer to the proof of~\cite[Theorem 1.3.4\,(iii)]{KY24}, where it is proved that $\H^0(K_w,\Fil^-(M_{\tilde{f}}))$ is finite, whose proof also works if we replace $K_w$ by $L_u$. So $\H^0(L_u,\Fil^-(M_{\tilde{f}}))=\H^0(L_{u,\infty}, \Fil^-(A_{\tilde{f}}))$ is finite.
		
		Since $\H^0(K_{w,\infty},\Fil^-(A_\epsilon))$ is a subgroup of $\H^0(L_{u,\infty},\Fil^-(A_\epsilon))$, we obtain a bound of the desired sort.
		
		The rest of the proof is identical to that of~\cite[Lemma 2.2.7]{How2004}.
	\end{proof}
	
	From here, one could argue similarly as in~\cite[Theorem 3.0.5]{KY24} to obtain the desired conditional result of a one-side divisibility of the Heegner Point Main Conjecture.
	
	\begin{theorem}[Analog of~\cref{Hgstr}]
		Assume $\H^0(K,\ol{\rho_{\tilde{f}}\otimes \epsilon})=0$. Suppose there is a Kolyvagin system $\kappa\in\KS(\bT_\epsilon,\cF_{\Lambda,\epsilon},\sL_{\tilde{f},\epsilon})$ with $\kappa_1\ne 0$. Then $\H^1_{\cF_{\Lambda,\epsilon}}(K,\bT_\epsilon)$ has $\Lambda$-rank one, and there is a finitely generated torsion $\Lambda$-module $M$ such that\begin{enumerate}
			\item[(i)] $\cX\sim\Lambda\oplus M\oplus M,$
			\item[(ii)] $\Char_\Lambda(M)$ divides $\Char_\Lambda(\H^1_{\cF_{\Lambda,\epsilon}}(K,\bT_\epsilon)/\Lambda\kappa_1).$
		\end{enumerate}
	\end{theorem}

	That there is indeed a non-trivial Kolyvagin system for the Heegner pair $(\tilde{f},\chi_\epsilon)$ follows from~\cite{JLZ21} (generalizing the results of~\cite{CastellaHsieh}) combined with non-vanishing results of the BDP $p$-adic $L$-function, as is explained in~\cite[Theorem 3.0.6]{KY24}.
	
	\begin{theorem}\label{nontorHeeg}
		Assume $\tilde{f}$ has weight $2$. Assume $p>2$ and that $p\nmid N'$. Then there exists a Kolyvagin system $\kappa^\Heeg\in\KS(\bT_\epsilon,\cF_{\Lambda,\epsilon},\sL_{\tilde{f},\epsilon})$ coming from the Heegner pair $(\tilde{f},\chi_\epsilon)$ such that $\kappa^\Heeg_1\in\H^1_{\cF_{\Lambda,\epsilon}}(K,\bT_{\epsilon})$ is non-torsion.
	\end{theorem}
	\begin{proof}
		For simplicity we only consider the weight $2$ case (so $a=b=k=j=0$ in~\cite{JLZ21}), which will be sufficient for our application. Again as in~\cite[Theorem 3.0.6]{KY24}, we need to work with a canonical lattice $\tilde{T}$ that is generally different from our chosen $T_{\tilde{f}}$ and may not satisfy the hypothesis (H.1) in~\cref{potKoly}. We use a tilde to denote relevant modules for the canonical lattice. One then replaces the appeal to~\cite{CastellaHsieh} by~\cite{JLZ21} for the construction of Heegner cycles in~\cite[Theorem 3.0.6]{KY24}. To do so, we need a slight generalization of~\cite{JLZ21}, as we now explain. In \textit{loc. cit.} section 3, they only considered the cycles $\Delta_{\phi_m}$ coming from pairs $(A_m,\phi_m)$, which only sees the `class field' direction. In other words, they only considered the variation along the class fields $K[p^m]$ and their Heegner cycle $z_{(\tilde{f},\chi_\eps)}$ is only the base point of a Kolyvagin system. To build the Kolyvagin system, we also need Heegner cycles that vary in different `Kolyvagin' levels. To do this, we simply consider cycles $\Delta_{np^m}$ coming from pairs $(A_{\cO_{np^m}},\phi_{\cO_{np^m}})$ and consider their images $z_{\text{\'e}t,np^m}^{[\tilde{f},0]}\in\H^1(K[np^m],\tilde{T}_\eps)$ as in~\cite[Proposition 3.5.2]{JLZ21} (here we take the ring $E$ in \text{loc. cit.} to be $\bZ_p$ instead of $\bQ_p$) under what essentially is the $p$-adic Abel--Jacobi map in~\cite[Section 4.2]{CastellaHsieh}. These images are nothing but the $z_{f,\cO_{np^m}}$ in~\cite[eq. (4.2)]{CastellaHsieh} where their $\chi$ is our $\chi_\eps$ that essentially sees the nebentypus of $\tilde{f}$.

		That the Heegner cycles can be turned into a Kolyvagin system is now discussed in~\cite[Lemma 4.12]{LV}, except that we take $\chi=\chi_\eps$ instead of the trivial character in \textit{op. cit.} section 4.1, and is modified in~\cite{KY24} to treat the Eisenstein case. By an abuse of notation, we will use the same notations from~\cite{LV}, omitting the $\chi$. From~\cite[Section 4]{CastellaHsieh}, the cycles with nebentypus behaves in the same way as those with trivial nebentypus. Most results can either be proved directly or by going up to the field $L_u$ so we have access to the geometric object $\tilde{A}_{\tilde{f}}$. We now explain the differences in our situation, which only appear in the case $v\mid p$. Note that here we are only working with weight $2$ forms so working with a modification of~\cite{How2004} would be sufficient, but we would also like to put ourselves in some general contexts. 
		
		First, as in~\cite[Theorem 3.0.6]{KY24}, the restriction map\[
		\H^1(K[n],\tilde{\bT}_\eps/I_n\tilde{\bT}_\eps)^{\cG(n)}\leftarrow\H^1(K,\tilde{\bT}_\eps/I_n\tilde{\bT}_\eps)\]can fail to be an isomorphism. We also introduce a $p$-power $p^N$ independent of $n\in\sN$ so that $p^N\tilde{\kappa}$ has a unique pre-image in $\H^1(K,\tilde{\bT}_\eps/I_n\tilde{\bT}_\eps)$, still denotes by $\kappa_1.$
		
		Secondly, in the study of the diagram (22) in~\cite{LV}, the injectivity of the right vertical map follows from the modified arguments in~\cite[Theorem 3.0.6]{KY24} since one only uses $\tilde{T}^-_\eps/I_n$ is finite for $n\ne 1$ (as usual, the case $n=1$ can be studied seperately).
		
		Thirdly, as is already mentioned in~\cite[Theorem 3.0.6]{KY24}, in general the surjectivity of the trace map $\tr_{F[n]/F}$ can fail, in which case one needs to work with a Kolyvagin system `up to a $p$-power'. This error can be ignored if one is able to show the algebraic and analytic $\mu$-invariants of $\tilde{f}\otimes\eps$ are equal, for example, in the Eisenstein case (where they both vanish). Since this work may also suggest some hints for the residually irreducible case, we also discuss the special case where the trace map is indeed expected to be surjective (e.g., if the characters are non-trivial in the Eisenstein case, or the `non-anomalous case' in the non-Eisenstein case) where one would get a fully functioning Kolyvagin system and hence one divisibility of the Heegner point Main Conjecture.
		
		Case (I): We first put ourselves in the situation where one has the vanishing of $\H^0(L_u,\tilde{A}_{\tilde{f}}^-)$. From the local Euler characteristic formula, one can show $\H^1(L_u,\tilde{A}_\eps^-)=\H^1(L_u,\tilde{A}_{\tilde{f}}^-)=0$. Thus the arguments in~\cite{LV} imply the vanishing of $\H^1(L[n]_{\tilde{u}}/L_u,\H^0(L[n]_{\tilde{u}},\tilde{A}^-_\epsilon))$. Note that $
		\H^1(K_v,\tilde{A}^-_\epsilon)$ maps into $H^1(L_u,\tilde{A}^-_\epsilon)$ with kernel $\H^1(L_u/K_v,\H^0(L_u,\tilde{A}^-_\epsilon))$, which is $0$ if we choose $L$ according to~\cref{goodfield}, since $\tilde{A}^-_\epsilon$ has $p$-power order. It follows that one gets the desired vanishing and $\tr_{F[n]/F}$ is surjective.
		
	Case (II): Now in the general case when $\tr_{F[n]/F}$ is not expected to be surjective, one can use the same trick as in~\cite[Theorem 3.0.6]{KY24} to use another fixed $p$-power $p^t$ (independent of $n$) to kill the cokernel of the trace map. This is possible since the codomain $\H^0(K_v,\tilde{\bA}_\eps^-)$ (in fact, $\H^0(K_v,\tilde{\bA}_\eps^-)=\H^0(K_v,\tilde{\bA}_{\tilde{f}}^-)$) is finite by the arguments in~\cite[Proposition 1.3.4]{KY24} whose proof also works in one replaces $K_v$ by $L_u$.
		
		The rest of the modification of~\cite[Lemma 4.12]{LV} is the same as in~\cite[Theorem 3.0.6]{KY24}. Eventually one passes to our chosen lattice $T_{\tilde{f}}$ from the canonical lattice $\tilde{T}$ to get a desired Kolyvagin system for $T_\eps$.
	\end{proof}
	
	We now arrive at a one-side divisibility of the Heegner Point Main Conjecture under mild hypotheses.
	\begin{theorem}\label{Hgstrpo}
		Assume $p>2$ and that $p\nmid N'$. Assume that $\H^0(K,\ol{\rho_{\tilde{f}}\otimes \epsilon})=0$. Then $\H^1_{\cF_{\Lambda,\epsilon}}(K,\bT_\epsilon)$ has $\Lambda$-rank one, and there is a finitely generated torsion $\Lambda$-module $M$ such that\begin{enumerate}
			\item[(i)] $\cX\sim\Lambda\oplus M\oplus M,$
			\item[(ii)] $\Char_\Lambda(M)$ divides $\Char_\Lambda(\H^1_{\cF_{\Lambda,\epsilon}}(K,\bT_\epsilon)/\Lambda\kappa_1).$
		\end{enumerate}
	\end{theorem}
	
	\subsection{Equivalence between two Iwasawa Main Conjectures}\label{PRreg}
	In this section we will follow~\cite[Appendix A]{Cas17} to show the equivalence between the Greenberg's Main Conjecture and the Heegner Point Main Conjecture. As in the previous subsection, we will need to replace the appeal to the big logarithm map (or the \textit{Perrin-Riou's regulator map}) in~\cite{CastellaHsieh} by the one for Heegner pairs in~\cite{JLZ21}. We continue to assume $p\nmid N'$. We first recall some terminologies. 
	
	\begin{definition}[Greenberg's Selmer groups]
		The Greenberg's Selmer group for the Heegner pair $(\tilde{f},\chi_\epsilon)$, denoted by $\H^1_{\cF_\Gr}(K,M_\epsilon)$ is defined as \[
		\H^1(K^\Sigma/K,M_\epsilon)\to \prod_{w\in\Sigma, w\nmid p}\H^1(K_w,M_\epsilon)\times \H^1(K_{\ol v},M_\epsilon)
		\]
		Denote by $\fX$ the Pontryagin dual of $\H^1_{\cF_\Gr}(K,M_\epsilon)$.
	\end{definition}
	
	\begin{definition}[unramified Selmer groups]
		The unramified Selmer group for the Heegner pair $(\tilde{f},\chi_\epsilon)$, denoted by $\H^1_{\cF_\ur}(K,M_\epsilon)$ is defined as \[
		\H^1(K^\Sigma/K,M_\epsilon)\to \prod_{w\in\Sigma, w\nmid p}\H^1(K_w,M_\epsilon)\times \H^1(I_{\ol v},M_\epsilon)^{G_{\ol v}/I_{\ol v}}
		\]
	\end{definition}
	
	These two Selmer groups are identified in the situation of~\cite{CGLS} under their running hypothesis that $\theta|_{G_{\ol v}}\ne\mathbf{1}$ or $\omega$. In~\cite{KY24} one needs to consider the unramified Selmer groups. However, as is already recalled in~\cref{Iwasawago}, these two Selmer groups generate the same characteristic $\Lambda$-ideal (see~\cite[Lemma 1.3.6]{KY24}). In fact, following Theorem 3.0.7 in \textit{op.\ cit.}, in showing the equivalence of the Iwasawa Main Conjectures, we will stick to the Greenberg's Selmer groups.
	
	On the analytic side, there is a BDP $p$-adic $L$-function $\cL_\epsilon\defeq \cL_p(\tilde{f})(\chi_\epsilon)\defeq \chi_\epsilon(\cL_p(\tilde{f}))$ corresponding to a Heegner pair $(\tilde{f},\chi_\epsilon)$. We mention that when $f=\tilde{f}\otimes\eps$, from the interpolation property (\cite[Theorem 5.13]{BDP13}), there is a relation\[
	\cL_p(\tilde{f},\chi_\epsilon(\cdot))=\cL_p(f,\cdot).
	\] We have the following explicit form of~\cref{JLZmain}.
	
	\begin{theorem}[The explicit reciprocity law]
		Assume $p=v\ol v$ is split in $K$. Then there is an injective $\Lambda$-linear map (the Perrin-Riou's regulator map) \[\cL_+:\H^0(K_v,\Fil^+\bT_\epsilon)\Lambda^\ur\to\Lambda^\ur\] with finite cokernel such that\[
		\cL_+(z_{(\tilde{f},\chi),\infty})=-\cL_\epsilon\cdot \sigma_{-1,v},
		\]
		where $\sigma_{-1,v}\in\Gamma$ is an element of order two.
	\end{theorem}
	\begin{proof}
		This follows from the specialization to the Heegner pair $(\tilde{f},\chi_\epsilon)$ of~\cite[Theorem B]{JLZ21}, as in~\cite{CastellaHsieh}, since we assume $p\nmid N'$. Here the class $z_{(\tilde{f},\chi),\infty}$ is the specialization of the class $\cores_{K_1/K}(z_{\cF,\infty})$ in Theorem 5.3.1 in \textit{op. cit.} to the point $(\tilde{f},\alpha,\chi)$. Note that we do not specialize to a finite level $m$ here. In particular, if we specialize our $z_{(\tilde{f},\chi),\infty}$ to the $m=1$ case, we will get $z_{\text{\'e}t}^{[f_\alpha,0]}$. 
	\end{proof}
	
	Now the equivalence between the Iwasawa Main Conjectures follows exactly in the same way from~\cite[Appendix A]{Cas17} (the local condition $\Gr$ there is the ordinary local condition for us). One also needs a careful comparison of the class $z_{\tilde{f},\chi,\infty}$ and $\kappa_1^\Hg$ as is done in~\cite[Remark 4.1.3]{CGLS}. We consider their projection to $\H^1(K,T_\eps)$. From the above theorem, we see $z_{(\tilde{f},\chi),\infty}$ projects to $z_{\text{\'e}t}^{[\tilde{f}_\alpha,0]}$. On the other hand, from~\cref{nontorHeeg}, $\kappa_1^\Hg$ projects to $\cores_{K_1/K}z_{\text{\'e}t,1}^{[\tilde{f},0]}=z_{\text{\'e}t}^{[\tilde{f},0]}$. Now from~\cite[Proposition 3.5.5]{JLZ21}, $z_{\text{\'e}t}^{[\tilde{f}_\alpha,0]}$ and $z_{\text{\'e}t,}^{[\tilde{f},0]}$ differ by a unit. Thus, we have the following.
	
	\begin{proposition}
		Assume $p=v\ol v$ splits in $K$ and that $p\nmid N'$. Assume that $\H^0(K,\ol{\rho_{\tilde{f}}\otimes \epsilon})=0$. Then the following statements are equivalent:\begin{enumerate}
			\item[(HPMC)] Both $\H^1_{\cF_{\Lambda,\epsilon}}(K,\bT_\epsilon)$ and $\cX=\H^1_{\cF_{\Lambda,\epsilon}}(K,M_\epsilon)^\vee$ have $\Lambda$-rank one, and the divisibility \begin{equation*}
				\Char_\Lambda(\cX_{\tors})\supset \Char_\Lambda(\H^1_{\cF_{\Lambda,\epsilon}}(K,\bT_\epsilon)/\Lambda z_{(\tilde{f},\chi),\infty})^2
			\end{equation*}
			holds in $\Lambda_{ac}$.
			\item[(GrMC)] Both $\H^1_{\cF_{\Gr}}(K,\bT_\epsilon)$ and $\fX_f=\H^1_{\cF_{\Gr}}(K,M_\epsilon)^\vee$ are $\Lambda$-torsion, and the divisibility\begin{equation*}
				\Char_\Lambda(\fX)\Lambda^{\nr}\supset(\cL_\epsilon)
			\end{equation*}
			holds in $\Lambda^{\nr}$.
		\end{enumerate}Moreover, the same result holds for the opposite divisibilities.
	\end{proposition}
	
	\subsection{Comparing Iwasawa invariants}\label{Iwainv}
	In this section, we compare the Iwasawa $\mu$- and $\lambda$-invariants of the Greenberg's Selmer groups and the BDP $p$-adic $L$-function. In fact, since these definitions do not see the reduction type of $f$ at $p$, all the arguments are exactly the same as in the ordinary case in~\cite{KY24}. The equalities between the Iwasawa invariants of both sides combined with the one-side divisibility obtained in the previous subsection, force the equality in the Iwasawa Main Conjectures to hold.
	
	More precisely, the analysis on the algebraic side is exactly as in~\cite{KY24}. On the analytic side, recall the relation $f=\tilde{f}\otimes\epsilon$ where $\tilde{f}$ has nebentypus $\epsilon^{-2}$. Since $p$ is an Eisenstein prime for $f$, it is also an Eisenstein prime for $\tilde{f}$. In particular, there is an exact sequence\[
	0\to\F(\phi)\to\ol\rho_{\tilde{f}}\to \F(\psi) \to 0,
	\]
	from which we obtain a congruence between the non-constant terms of $\tilde{f}$ and a certain Eisenstein series $G^{\phi,\psi,N'}$ (see~\cite[Remark 32, Theorem 34]{Kri16}). In other words, $\tilde{f}$ is of partial Eisenstein descent in the sense of \textit{op.\ cit.}. 
	
	Similarly, there is also an exact sequence\[
	0\to\F(\phi\epsilon)\to\ol\rho_f\to\F(\psi\epsilon)\to 0,
	\]
	from which we obtain a congruence between the non-constant terms of $\tilde{f}\otimes\epsilon$, which is another newform, and a certain Eisenstein series $G^{\phi\epsilon,\psi\epsilon,N}$. As in~\cite[section 2.2]{CGLS}, this yields a congruence between the BDP $p$-adic $L$-functions \[\cL_\epsilon\equiv(\cL_{G^{\phi\epsilon,\psi\epsilon,N}})^2(\text{mod }p\Lambda^\ur),\] and moreover for the Katz $p$-adic $L$-functions $\cL_{\phi\epsilon}$ and $\cL_{\psi\epsilon}$ an equality \[\cL_{G^{\phi\epsilon,\psi\epsilon,N}}=\mathcal{E}^\iota_{\phi\epsilon,\psi\epsilon}\cL_{\phi\epsilon}\] for a certain factor introduced in\text{loc.\ cit.}. Since $(\phi\epsilon)(\psi\epsilon)=\mathbf{1}$, the same arguments in~\cite[Theorem 2.2.2]{CGLS} show that the Iwasawa invariants on the analytic side are compatible with those on the algebraic side. 
	
	Now as in~\cite{KY24}, using the Iwasawa Main Conjectures for the characters proved by Rubin in~\cite{Rubin1991} (together with~\cite{CW78}. See also~\cite[III.1.10]{deShalit}), one obtains the following theorem.
	
	\begin{theorem}\label{IMCpo}
		Assume we are in Case (I) of~\cref{twists}, i.e., $p\nmid N'$. Assume that $p=v\ol v$ splits in $K$. Assume that $\H^0(K,\ol{\rho_{\tilde{f}}\otimes \epsilon})=0$. Then \[\mu(\fX)=\mu(\cL_\epsilon)=0\text{ and }\lambda(\fX)=\lambda(\cL_\epsilon).\]
		Consequently,\[
		\Char_\Lambda(\fX)\Lambda^\ur=(\cL_\epsilon)\text{ holds in }\Lambda^\ur,
		\]or equivalently,\[
		\Char_\Lambda(\cX_\tors)=\Char_\Lambda(\H^1_{\cF_\Lambda}(K,\bT)/\Lambda\kappa_\infty)^2\text{ holds in }\Lambda.
		\]
	\end{theorem}
	
	\begin{remark}
		\begin{itemize}
			\item As in~\cite[Theorem 3.0.8]{KY24}, the proof of the Main Conjectures heavily rely on the vanishing of the $\mu$-invariants, especially in the situation where $N,t>0$.
			\item As in~\cite[Remark 3.0.9]{KY24}, the assumption $\H^0(K,\ol{\rho_{\tilde{f}}\otimes \epsilon})=0$ can be removed from Ribet's lemma and Perrin-Riou's formula.
		\end{itemize}
		
	\end{remark}

	\subsection{Control theorems for potentially good ordinary reduction}
	In this subsection, we prove the control theorems for potentially good ordinary reduction. The proof is similar to that of~\cref{controlgo}, mainly following~\cite{Greenberg1999}. In particular, we assume $V_\epsilon$ corresponds to an elliptic curve of potentially ordinary reduction at $p$, so that $\tilde{f}$ is a weight $2$ newform. In particular, when $E$ is an elliptic curve, we have $T_\epsilon=T_pE$, $V_\epsilon=T_pE\otimes \bQ_p$ and $A_\epsilon=E[p^\infty]$.

	We begin by noting that the proof will only be different at places above $p$. More precisely, as in the good ordinary case, it suffices to show the map $r_w$ has finite kernel for each place $w$ of $K$, and for $w\nmid p$ the argument is the same as those in~\cref{controlgo}. On the other hand, the study of $\ker(r_v)$ and $\ker(r_{\ol v})$ in~\cite[Lemma~3.4]{Greenberg1999} makes use of the kernel of reduction $E[p^\infty]\to \tilde{E}[p^\infty]$ where $\tilde{E}$ is the reduction of $E$ at $v$ and $\ol v$ respectively, which becomes more mysterious in if $E$ has additive reduction at $p$. Luckily, for the $p$-converse theorems, we do not need to understand the kernels completely, and it suffices to show they are finite. In fact, we will choose an extension which trivializes $\epsilon$ as before, so that $V_\epsilon$ is identified with $V_{\tilde{f}}$ over this extension and one could apply the arguments from the ordinary case. It then remains to show climbing up the field only introduces finite errors.

	\begin{theorem}[Control theorem]\label{controlpo}
		With notations from~\cref{controlgo}, in the potentially good ordinary but not ordinary case, $\ker(r_v)$ and $\ker(r_{\ol v})$ are finite.
	\end{theorem}
	\begin{proof}
		In our notation, for $w=v,\ol v$, the map $r_w$ reads as\[
		r_w\colon\frac{\H^1(K_w,A_\epsilon)}{\H^1_\BK(K_w,A_\epsilon)}\to\frac{\H^1(K_w,M_\epsilon)}{\H^1_\BK(K_w,M_\epsilon)}.
		\]As in~\cref{tword}, the Bloch--Kato local condition for $A$ agrees with the twisted ordinary local condition coming from a Heegner pair $(\tilde{f},\chi)$. Similarly, the Bloch--Kato local condition for $M_E$ can be identified with a twisted ordinary local condition. Therefore the map $r_w$ is identified with{\small\[
			\frac{\H^1(K_w,A_\epsilon)}{p_*(\im(H^1(K_w,\Fil^+(V_\epsilon))\to\H^1(K_w,V_\epsilon)))}\to\frac{\H^1(K_w,M_\epsilon)}{\im(H^1(K_w,\Fil^+(M_\epsilon)))\to\H^1(K_w,M_\epsilon))}.
			\]}
		where $p_*$ denotes propagation from $\H^1(K_w,V_\epsilon)$ to $\H^1(K_w,A_\epsilon)$. From now on, we simply write $\im(\H^1(-,\Fil^+(-)))$ for images and ignore the codomain when it is clear from the context.
		
		As before, let $L_u/K_w$ be an extension that trivializes $\epsilon$. We first assume $\tilde{f}$ is of good reduction. Then as in the good ordinary case (see~\cite[Lemma 3.4]{Greenberg1999}), for $G\defeq\Gal(L_u/K_v)$, the `restricted' map \[r_u\colon \frac{\H^1(L_u,A_\epsilon)^G}{p_*(\im(H^1(L_u,\Fil^+(V_\epsilon))))^G}\to\frac{\H^1(L_u,M_\epsilon)^G}{\im(H^1(L_u,\Fil^+(M_\epsilon))))^G}\]
		has finite kernel. 
		
		From snake lemma applied to the maps defining $\ker(r_w)$ and $\ker(r_u)$, to get finiteness of $\ker(r_w)$, it suffices to show finiteness of the map\[
		r\colon\frac{\H^1(K_w,A_\epsilon)}{p_*(\im(H^1(K_w,\Fil^+(V_\epsilon))))}\to\frac{\H^1(L_u,A_\epsilon)^G}{p_*(\im(H^1(L_u,\Fil^+(V_\epsilon))))^G}.
		\]
		
		Again from snake lemma applied to the commutative diagram\begin{center}
			\begin{tikzcd}
				\ker(t)\ar[r]\ar[d]&\ker(s)\ar[r]\ar[d]&\ker(r)\ar[d]
				\\	p_*(\im(H^1(K_w,\Fil^+(V_\epsilon))))\ar[r]\ar[d,"t"]& \H^1(K_w,A_\epsilon) \ar[r]\ar[d,"s"]& \frac{\H^1(K_w,A_\epsilon)}{p_*(\im(H^1(K_w,\Fil^+(V_\epsilon))))}\ar[d,"r"] \\
				p_*(\im(H^1(L_u,\Fil^+(V_\epsilon))))^G \ar[r]\ar[d]& \H^1(L_u,A_\epsilon)^G  \ar[r]&\frac{\H^1(L_u,A_\epsilon)^G}{p_*(\im(H^1(L_u,\Fil^+(V_\epsilon))))^G}\\
				\coker(t)&&,
		\end{tikzcd}\end{center}
		it suffices to show both $\ker(s)$ and $\coker(t)$ are finite. Now from the inflation-restriction exact sequence, $\ker(s)=\H^1(\Gal(L_u/K_v),A_\epsilon^{G_{L_u}})$, which is finite since $\Gal(L_u/K_v)$ and $A_\epsilon^{G_{L_u}}=E(L_u)[p^\infty]$ are both finite (the latter by~\cite[Proposition 1.3.4\,(iii), weight $2$ case]{KY24}. One can replace their $K_w$ by $L_u$).
		
		It remains to study $\coker(t)$. First note that $\ker(t)$ is finite because it injects into $\ker(s)$. We will show that $p_*(\im(H^1(K_w,\Fil^+(V_\epsilon))))$ and $p_*(\im(H^1(L_u,\Fil^+(V_\epsilon))))^G$ have the same $\bZ_p$
		-corank, from which the finiteness of $\coker(t)$ follows.
		
		Since the restriction map $\H^1(K_w,\Fil^+(V_\epsilon))\to\H^1(L_u,\Fil^+(V_\epsilon))$ is an isomorphism of vector spaces (the composition of restriction and corestriction is a scalar multiplication. See for example~\cite[Corollary 1.5.7]{NSW2.3}), they have the same $\bQ_p$-dimensions. We claim that propagation takes $\bQ_p$-dimensions to $\bZ_p$-coranks, i.e., if a subspace of $\H^1(K_w,V_\epsilon)$ has $\bQ_p$-dimension $d$, then its propagation into $\H^1(K_w,V_\epsilon)$ has $\bZ_p$-corank $d$. 
		
		Let $V'$ be a $d$-dimensional subspace of $\H^1(K_w,V_\epsilon)$, and let $p_*(V')$ be its propagation to $\H^1(K_w,A_\epsilon)$. Then there is an exact sequence\[
		0\to \ker(p_*) \to V'\to p_*(V') \to 0
		\]
		where $\ker(p_*)$ is a subspace of $\H^1(K_w,T_\epsilon)$, and is hence a finitely generated $\bZ_p$-module. Denoting the $\bZ_p$-rank, the $\bQ_p$-dimension and the $\bZ_p$-corank of the three terms in the above exact sequence by $r,d,c$ respectively, one sees that there must be an identity $r=d=c$.
		
		The same is true for subspaces of $\H^1(L_u,V_\epsilon)^G$ since propagation commutes with invariants. Therefore, from the isomorphism in the previous paragraph, both $p_*(\im(H^1(K_w,\Fil^+(V_\epsilon))))$ and $p_*(\im(H^1(L_u,\Fil^+(V_\epsilon))))^G$ have $\bZ_p$-corank $d$. Thus $\coker(t)$ is finite.
		
		Note that all we use is that over the extension $L_u$, the map $r_u$ has finite kernel, and that the restriction maps only introduce some finite errors. Thus the same arguments work when $\tilde{f}$ has multiplicative reduction, by the discussion of the multiplicative reduction in~\cite[section 3]{Greenberg1999}, except that the appeal to~\cite[1.3.4\,(ii)]{KY24} is now modified as in~\cref{Specialize} to fit the multiplicative case.
	\end{proof}
	
	\subsection{The $p$-converse theorems for potentially good ordinary reduction}
	Combining everything above, we arrive at the main theorem in this section.
	
	\begin{theorem}\label{pCmain}
		Let $E$ be an elliptic curve defined over $\bQ$ and let $p>2$ be a prime of potentially good ordinary reduction for $E$. Assume that $E[p]$ is reducible. Let $r\in\{0,1\}$. Then \[
		\corank_{\bZ_p}\Sel_{p^\infty}(E/\bQ)=r\Rightarrow \ord_{s=1}L(E,s)=r,
		\]and so $\rk_\bZ E(\bQ)=r$ and $\Sha(E/\bQ)<\infty$.
	\end{theorem}
	\begin{proof}
		Let $(\tilde{f},
		\chi_\epsilon)$ be the Heegner pair associated to $E$. As in the good ordinary case, the proof relies on the choice of an imaginary quadratic field $K$. The only difference is that in addition to the hypotheses $K$ should verify, we further require that the prime divisors of $\cond(\epsilon)$ are split in $K$. As in the proof of~\cite[Theorem~5.2.1]{CGLS}, this only impose a finite number of congruence conditions on $D_K$, so the existence of $K$ is still guaranteed. By replacing the appeal to~\cref{controlgo} (\textit{resp.~\cref{Hgstr},~\cref{IMC}}) to~\cref{controlpo} (\textit{resp.~\cref{Hgstrpo},~\cref{IMCpo}}), the rest of the proof is exactly the same as that in~\cref{pCgo}.
	\end{proof}
	
	\subsection{Applications to Goldfeld's Conjecture}
	As in~\cite[remark 4.1.2]{KY24}, our $p$-converse theorem has the following applications to Goldfeld's Conjecture in quadratic twists families with a $3$ isogeny where $3$ is a prime of potentially good ordinary reduction.
	
	\begin{corollary}\label{Goldfeld}
		We can obtain better proportions of quadratic twists of (algebraic and analytic) rank~$1$ in~\cite[Theorem~2.5]{BKLOS} for a fixed elliptic curve, in particular a lower bound of $\frac{5}{12}=41.66\ldots\,\%$ in the most advantageous cases (for example, when it's the curve having Cremona label $19a3$). For the twists $E_d$ of the curve $E=19a3$, earlier results in~\cite{CGLS} and~\cite{KY24} only apply when $d\equiv 1,2 \pmod{3}$, covering at least $\frac{3}{8}+\frac{3}{8}=\frac{3}{4}$ of the proportion of the $3$-Selmer rank $1$ twists. Our result also cover the twists with $d\equiv 0 \pmod{3}$ ($\frac{2}{8}$ of all) that correspond to additive reduction (necessarily potentially good ordinary because the original curve is). Consequently, in this family, all curves with $3$-Selmer ranks $0$ or $1$ have algebraic and analytic rank equal to the $3$-Selmer ranks, so at least $\frac{5}{12}$ (rank $1$) $+$ $\frac{1}{4}$ (rank $0$) $=\frac{2}{3}$ twists satisfy the BSD rank conjecture. This provides strong evidence for the Goldfeld conjecture.
	\end{corollary}

	\bibliographystyle{amsalpha}
	\bibliography{references}

\providecommand{\bysame}{\leavevmode\hbox to3em{\hrulefill}\thinspace}
\providecommand{\MR}{\relax\ifhmode\unskip\space\fi MR }
\providecommand{\MRhref}[2]{%
  \href{http://www.ams.org/mathscinet-getitem?mr=#1}{#2}
}
\providecommand{\href}[2]{#2}
\begin{thebibliography}{BKLOS19}

\bibitem[AI19]{AI19}
Fabrizio Andreatta and Adrian Iovita, \emph{Katz type p-adic l-functions for
  primes p non-split in the cm field}, arXiv: Number Theory (2019).

\bibitem[BC09]{BrinonConrad}
Olivier Brinon and Brian Conrad, \emph{{CMI Summer School Notes on $p$-adic
  Hodge theroy}}, 2009,
  \url{https://claymath.org/sites/default/files/brinon_conrad.pdf}.

\bibitem[BCK21]{BCK21}
Ashay Burungale, Francesc Castella, and Chan-Ho Kim, \emph{{A proof of
  Perrin-Riou's Heegner point main conjecture}}, Algebra Number Theory
  \textbf{15} (2021), no.~7, 1627--1653.

\bibitem[BDP13]{BDP13}
Massimo Bertolini, Henri Darmon, and Kartik Prasanna, \emph{{Generalized
  {H}eegner cycles and {$p$}-adic {R}ankin {$L$}-series}}, Duke Math. J.
  \textbf{162} (2013), no.~6, 1033--1148, With an appendix by Brian Conrad.

\bibitem[BKLOS19]{BKLOS}
Manjul Bhargava, Zev Klagsbrun, Robert~J. Lemke~Oliver, and Ari Shnidman,
  \emph{{{\(3\)}}-isogeny {Selmer} groups and ranks of abelian varieties in
  quadratic twist families over a number field}, Duke Math. J. \textbf{168}
  (2019), no.~15, 2951--2989.

\bibitem[Cas15]{Cas15}
Francesc Castella, \emph{On the exceptional specializations of big heegner
  points}, Journal of the Institute of Mathematics of Jussieu \textbf{-1}
  (2015).

\bibitem[Cas17]{Cas17}
\bysame, \emph{{{$p$}-adic heights of Heegner points and Beilinson-Flach
  classes}}, Journal of the London Mathematical Society \textbf{96} (2017),
  156--180.

\bibitem[CGLS22]{CGLS}
Francesc Castella, Giada Grossi, Jaehoon Lee, and Christopher Skinner,
  \emph{{On the anticyclotomic Iwasawa theory of rational elliptic curves at
  Eisenstein primes}}, Inventiones mathematicae \textbf{227} (2022), 517--580.

\bibitem[CGS23]{CGS}
Francesc Castella, Giada Grossi, and Christopher Skinner, \emph{{Mazur's main
  conjecture at Eisenstein primes}}, 2023,
  \url{https://doi.org/10.48550/arXiv.2303.04373}.

\bibitem[CH18]{CastellaHsieh}
Francesc Castella and Ming-Lun Hsieh, \emph{{Heegner cycles and {$p$}-adic
  {$L$}-functions}}, Math. Ann. \textbf{370} (2018), no.~1--2, 567--628.

\bibitem[Con]{Conrad}
Brian Conrad, \emph{{Semistable reduction for Abelian varieties}},
  \url{https://math.stanford.edu/~conrad/DarmonCM/2011Notes/SemistableReduction.pdf}.

\bibitem[Cor02]{Cor02}
Christophe Cornut, \emph{{Mazur’s conjecture on higher Heegner points}},
  Inventiones mathematicae \textbf{148} (2002), 495--523.

\bibitem[CW78]{CW78}
J.~Coates and A.~Wiles, \emph{{On $p$-adic $L$-functions and elliptic units}},
  Journal of the Australian Mathematical Society \textbf{26} (1978), no.~1,
  1–25.

\bibitem[dS87]{deShalit}
Ehud de~Shalit, \emph{Iwasawa theory of elliptic curves with complex
  multiplication : p-adic l functions}, 1987.

\bibitem[Gre99]{Greenberg1999}
Ralph Greenberg, \emph{Iwasawa theory for elliptic curves}, pp.~51--144,
  Springer Berlin Heidelberg, Berlin, Heidelberg, 1999.

\bibitem[Hid10]{Hida2010}
Haruzo Hida, \emph{The {I}wasawa {$\mu$}-invariant of {$p$}-adic {H}ecke
  {$L$}-functions}, Ann. of Math. (2) \textbf{172} (2010), no.~1, 41--137.

\bibitem[How04]{How2004}
Benjamin Howard, \emph{{The Heegner point Kolyvagin system}}, Compositio
  Mathematica \textbf{140} (2004), no.~1, 1439--1472.

\bibitem[JLZ21]{JLZ21}
Dimitar Jetchev, David Loeffler, and Sarah~Livia Zerbes, \emph{Heegner points
  in {C}oleman families}, Proceedings of the London Mathematical Society
  \textbf{122} (2021), no.~1, 124--152.

\bibitem[JSW17]{JSW2017}
Dimitar Jetchev, Christopher Skinner, and Xin Wan, \emph{{The Birch and
  Swinnerton-Dyer Formula for Elliptic Curves of Analytic Rank One}}, Cambridge
  Journal of Mathmatics \textbf{5} (2017), no.~3, 369--434.

\bibitem[Kat04]{Kato}
Kazuya Kato, \emph{{$p$-adic Hodge theory and values of zeta functions of
  modular forms}}, Ast\'{e}risque \textbf{295} (2004), 117--290.

\bibitem[Kri16]{Kri16}
Daniel Kriz, \emph{{Generalized Heegner cycles at Eisenstein primes and the
  Katz $p$-adic $L$-function}}, Algebra and Number Theory \textbf{10} (2016),
  no.~2, 309--374.

\bibitem[KS23]{KellerStoll2023}
Timo Keller and Michael Stoll, \emph{{Complete verification of strong BSD for
  many modular abelian surfaces over $\mathbf{Q}$}}, 2023,
  \url{https://arxiv.org/abs/2312.07307}.

\bibitem[KY24]{KY24}
Timo Keller and Mulun Yin, \emph{{On the anticyclotomic Iwasawa theory of
  modular forms of semistable reduction at Eisenstein primes}},
  \url{https://arxiv.org/abs/2402.12781}.

\bibitem[LV16]{LV}
Matteo Longo and Stefano Vigni, \emph{{Kolyvagin systems and Iwasawa theory of
  generalized Heegner cycles}}, Kyoto Journal of Mathematics (2016).

\bibitem[MR04]{MR04}
Barry Mazur and Karl Rubin, \emph{Kolyvagin systems}, American Mathematics
  Society (2004).

\bibitem[Mü24]{Müller24}
Katharina Müller, \emph{Kato’s main conjecture for potentially ordinary
  primes}, Glasgow Mathematical Journal (2024), 1–21.

\bibitem[Nek06]{NekovarSelmerComplexes}
Jan Nekov\'{a}\v{r}, \emph{Selmer complexes}, Ast\'{e}risque (2006), no.~310,
  viii+559.

\bibitem[NSW08]{NSW2.3}
Jürgen Neukirch, Alexander Schmidt, and Kay Wingberg, \emph{{Cohomology of
  Number Fields}}, second edition ed., Grundlehren der mathematischen
  Wissenschaften, Vol.~\textbf{323}, Springer-Verlag, 2008, version 2.3 from
  \url{https://www.mathi.uni-heidelberg.de/~schmidt/NSW2e/index-de.html}.

\bibitem[PR87]{PR87}
Bernadette Perrin-Riou, \emph{{Fonctions $L$ $p$-adiques, théorie d'Iwasawa et
  points de Heegner}}, Bulletin de la Société Mathématique de France
  \textbf{115} (1987), 399--456 (fre).

\bibitem[Rub91]{Rubin1991}
Karl Rubin, \emph{The ``main conjectures'' of {I}wasawa theory for imaginary
  quadratic fields}, Invent. Math. \textbf{103} (1991), no.~1, 25--68.

\bibitem[Shn21]{ShnidmanIMRN}
Ari Shnidman, \emph{Quadratic twists of abelian varieties with real
  multiplication}, Int. Math. Res. Not. \textbf{2021} (2021), no.~5,
  3267--3298.

\bibitem[Ski14]{Skinner}
Christopher Skinner, \emph{{Multiplicative reduction and the cyclotomic main
  conjecture for $\GL_2$}}, Pacific Journal of Mathematics \textbf{283} (2014),
  171--200.

\bibitem[SU14]{SU14}
Christopher Skinner and Eric Urban, \emph{The iwasawa main conjectures for
  gl2}, Inventiones mathematicae \textbf{195} (2014), 1--277.

\bibitem[Vat03]{Vat03}
V.~Vatsal, \emph{{Special values of anticyclotomic $L$-functions}}, Duke
  Mathematical Journal \textbf{116} (2003), no.~2, 219 -- 261.

\end{thebibliography}
	
\end{document}